\documentclass[12pt]{amsart}
\usepackage{amssymb,amsfonts,amsmath,amsthm}
\newcommand{\ol}{\overline}
\newcommand{\Ind}{\mbox{\bf I}}
\newcommand{\pp}{{\mathbb P}}
\newcommand{\mn}{{\mathfrak N}}
\newcommand{\mmm}{{\mathfrak M}}
\newcommand{\witi}{\widetilde}
\newcommand{\fracd}[2]{\frac {\displaystyle #1}{\displaystyle #2 }}
\newcommand{\zz}{{\mathbb Z}}
\newcommand{\nn}{{\mathbb N}}
\newcommand{\ee}{{\mathbb E}}
\newcommand{\rr}{{\mathbb R}}

\newcommand{\calf}{{\mathcal F}}
\newcommand{\calr}{{\mathcal R}}
\newcommand{\calp}{{\mathcal P}}
\newcommand{\cali}{{\mathcal I}}
\newcommand{\calg}{{\mathcal G}}
\newcommand{\bw}{{\bf W}}
\newcommand{\veps}{\varepsilon}
\newcommand{\beq}{\begin{eqnarray*}}
\newcommand{\feq}{\end{eqnarray*}}
\newcommand{\beqn}{\begin{eqnarray}}
\newcommand{\feqn}{\end{eqnarray}}
\newcommand{\as}{\mbox{a.s.}}
\newtheorem{theorem}{Theorem} \makeatletter
\@addtoreset{theorem}{section}\makeatother

\newtheorem{definition}[theorem]{Definition}
\newtheorem{lemma}[theorem]{Lemma}
\newtheorem{assume}[theorem]{Assumption}
\newtheorem{proposition}[theorem]{Proposition}
\newtheorem{corollary}[theorem]{Corollary}
\newtheorem{remark}[theorem]{Remark}

\begin{document}
\title{A random walk on $\zz$ with  drift driven by its occupation time at zero}
\author{Iddo Ben-Ari}
\address{Department of Mathematics, University of
California Irvine, 272 MSTB, Irvine, CA 92697-3875, USA}
\email{ibenari@math.uci.edu}
\author{Mathieu Merle}
\address{Department of Mathematics,
University of British Columbia,  121-1984 Mathematics
Road, Vancouver, BC, Canada V6T 1Z2}
\email{merle@math.ubc.ca}
\author{Alexander Roitershtein}
\thanks{I.~B.-A. and A.~R. thank the probability groups respectively at UBC
and at UC Irvine for the hospitality during visits in which
part of this work was carried out. Most of the work on this paper
was carried out while A.~R. enjoyed the hospitality of the
Department of Mathematics at UBC as a post-doc.}
\address{Department of Mathematics, Iowa State University, 472 Carver Hall,
Ames, IA 50011, USA}
\email{roiterst@iastate.edu}
\date{November 27, 2007}
\subjclass[2000]{Primary 60F17, 60F20, 60F5.}
\keywords{Limit theorems, renewal theorem, regular variation,
excursions of random walks, oscillating random walks,
invariance principle, Kakutani's dichotomy}
\begin{abstract}
We consider a nearest neighbor random walk on the one-dimensional
integer lattice with drift towards the origin determined by an
asymptotically vanishing function of the number of visits to zero.
We show the existence of distinct regimes according to the rate of
decay of the drift. In particular, when the rate is sufficiently
slow, the position of the random walk, properly normalized,
converges to a symmetric exponential law. In this regime, in
contrast to the classical case, the range of the walk scales
differently from its position.
\end{abstract}
\maketitle
\section{Introduction}
\label{intro} We consider a self-interacting random walk
$X:=(X_n)_{n\ge 0}$ on $\zz$ whose drift is a function of the number
of times it has already visited the origin. The random variable
$X_n$ represents the position of the walker at time $n\in\zz_+.$ We
assume that $|X_{n+1}-X_n|=1$ for all $n\ge 0$, that is $X$ is a
nearest neighbor model. Let $\eta_0$ be a positive integer and, for
$n\ge 1$, let
\begin{equation}
\label{etan}
\eta_n=\eta_0+\#\{i\in(0,n]:X_i=0\}.
\end{equation}
Thus, $\eta_n -\eta_0$ describes the number of visits of the walker to the origin by time $n$.
Let $\veps := (\veps_n)_{n \geq1}$ be a  sequence taking values in $[0,1)$.
For $x\in \zz$ and $l\in \nn$, let  $P_{(x,l)}^{\veps}$ denote
a measure on the nearest neighbor random walk paths defined as follows:
\begin{align}
\nonumber
&P_{(x,l)}^{\veps}(X_0=x,\eta_0=l)=1
\\
\label{eq:din}
&P_{(x,l)}^{\veps}(X_{n+1}=j|X_n = i, \eta_n=m) =
\left\{
\begin{array}{ll}
\frac 12&\mbox{if}~i=0~\mbox{and}~|j|=1
\\
\frac12 \bigl(1-\mbox{sign}(i)\veps_m\bigr)&\mbox{if}~i\ne 0~\mbox{and}~j-i=1.
\\
\frac12 \bigl(1+\mbox{sign}(i)\veps_m\bigr)&\mbox{if}~i\ne 0~\mbox{and}~j-i=-1.
\end{array}
\right.
\end{align}
Here $\mbox{sign}(x)$ is $-1,0,$ or $1$ according to whether $x$ is
a negative, zero, or, respectively, positive. The corresponding
expectation is denoted by $E_{(x,l)}^{\veps}.$
\par
To simplify the notation, we usually denote $P_{(0,1)}^{\veps}$ by
$P$ and $E_{(0,1)}^{\veps}$ by $E.$ If $\veps_n=0$ for all $n\geq 1,$ we
denote $P$ by $\pp,$ $E$ by $\ee,$ and refer to $X$ as the {\em
simple random walk on $\zz.$}
\par
We note that, unless $\veps$ is a constant sequence, $X$ is not a Markov chain.
However, the pairs $(X_n,\eta_n)_{n\ge 0}$ form a time-homogeneous Markov chain.
\par
Let $d_n=-\mbox{sign}(X_n)\veps_{\eta_n},$ and let
$\calf_n=\sigma(X_0,X_1,\dots,X_n)$ denote the $\sigma$-algebra
generated by the random walk paths up to time $n$. Then \beqn
\label{din_dyn} E \bigl( X_{n+1}-X_n | \calf_n \bigr) = d_n. \feqn
That is $d_n$ is the local drift of the random walk at time $n.$
Note that the drift is always toward the origin.
\par
The aim of this paper is to prove limit theorems for the model
described above in the case when $\lim_{n\to\infty}\veps_n=0$. If
the convergence is fast enough, the asymptotic behavior of $X$ is
similar to that of the simple random walk. In Theorem \ref{th:main}
we show that the functional central limit theorem holds when $n
\veps_n \to 0$ and that $P$ and $\pp$ are mutually absolutely
continuous if and only if $\sum_{n=1}^{\infty} \veps_n <\infty$. We
refer to this regime as supercritical. On the other hand, when
$\veps_n$ converges to $0$ slowly, the process exhibits a different
limiting behavior. This case is treated in Theorems
\ref{walk}--\ref{supt}. In particular, we show that when $\veps$ is
a regularly varying sequence converging to $0$ and satisfying $n
\veps_n \to\infty$, the position of the walk $X_n,$ properly
normalized, converges in distribution to a symmetric exponential
random variable. In this case, in contrast to the simple random
walk, the range of the walk up to time $n$ scales differently from
$X_n$. We call this regime subcritical. The critical regime, which
essentially corresponds to sequences satisfying $c_1 \le n \veps_n
\le c_2$ for some $0<c_1\le c_2 <\infty,$ is subject of future work.
\par
The above definition of the random walk was inspired by a branching
tree model arising in \cite{iperc} in context of the study of the
invasion-percolation on a regular tree. The scaling limit of this
branching tree is further studied in \cite{newp} and is closely
related to the critical regime of our model.
\par
It is well-known that there is a one-to-one correspondence between
discrete random trees and certain random walk paths (cf. \cite{LeGall}). 
Our model would correspond to a discrete random
binary tree with an infinite rightmost branch, a backbone, from
which emerge off-backbone trees. Such an off-backbone tree has a
single vertex at its first generation, from which emerges a
subcritical Galton-Watson tree, however the branching law in this
off-backbone tree depends on the height at which it branches off the
backbone.
\par
 Another related class
of random processes are oscillating random walks, namely
time-homogeneous Markov chains in $\rr^d$ with transition function
which depends on the position of the chain with respect to a fixed
hyperplane, cf. \cite{kemperman,borovkov}.
\par
We remark that the model can be interpreted as describing a gambler
(Sisyphus) who learns from his experience and adopts a new strategy
whenever a ruin event occurs. This paper intends to be a first step
towards a more general study of random walks in $\mathbb{Z}^d$ for
which the transition probabilities are updated each time the walk
visits a certain set. Another possible extension would be to
consider a random environment version of the random walk~$X.$
\par
The paper is organized as follows. The main results are collected in
Section~\ref{sec:results}. Some general facts about random walks and
regular varying sequence are recalled in Section~\ref{sec:prel}.
The proofs are contained in Section~\ref{sec:super} (supercritical
case) and Section~\ref{sec:sub} (subcritical case).
\section{Statement of main results}
\label{sec:results} This section presents the main results of this
paper. It is divided into two parts. The first is devoted to the
supercritical case while the second one covers the results for the
subcritical regime. Throughout the paper, unless it is explicitly
stated otherwise, we assume that the drift sequence $\veps$ is given
and consider the random walk $X$ under the measure $P$ defined
above.
\subsection{Supercritical Regime}
Let $C(\rr_+,\rr)$ be the space of continuous functions from $\rr_+$ into $\rr,$
equipped with the topology of uniform convergence on compact sets.
For a sequence of random variables $Z:=(Z_n)_{n\ge0}$ and each $n\ge 0$,
let $\cali^{Z}_n \in C(\rr_+,\rr)$ denote the following linear interpolation of $Z_{[nt]}$:
\begin{equation}
\label{esn} \cali^Z_n(t) = \fracd{1}{\sqrt{n}}\bigl(([nt]+1-nt)
Z_{[nt]} + (nt-[nt]) Z_{[nt]+1}\bigr).
\end{equation}
Here and henceforth $[x]$ denotes the integer part of a real number $x.$
\par
We say that $Z$ satisfies the invariance principle, if the sequence of processes
$\bigl(\cali_n^Z(t)\bigr)_{t \in \rr_+}$ converges weakly in
$C(\rr_+;\rr),$ as $n\to\infty,$ to the standard Brownian motion. We have:
\begin{theorem}~
\label{th:main}
\begin{enumerate}
\item  Assume that $\lim_{n\to\infty} n \veps_n =0$.  Then $X$ satisfies the invariance principle.
\item  The distribution of $X$ under $P$ is either equivalent or orthogonal to
the law $\pp$ of the simple random walk, according to whether $\sum_{n=1}^{\infty} \veps_n$ is finite or not.
\end{enumerate}
\end{theorem}
For the sake of comparison with the subcritical regime, we now state some
consequences of this result. Let
\begin{equation}
\label{maxima} M_n=\max_{i\leq n} X_i.
\end{equation}
We have:
\begin{corollary}~
\label{mainc}
\begin{enumerate}
\item
Assume that $\lim_{n\to\infty} n \veps_n =0$. Then $M_n/\sqrt{n}$
(respectively $|M_n|/\sqrt{n}$) converge in distribution, as
$n\to\infty,$ to $\sup_{0\leq t\leq 1} B_t$ (respectively to
$\sup_{0\leq t\leq 1} |B_t|$), where $B_t$ is the standard Brownian
motion.
\item
Assume that $\sum_{n=1}^{\infty} \veps_n <\infty$.
Then $\displaystyle \limsup_{n\to\infty} \frac{X_n}{\sqrt{2n\log\log n}}=1,$ $P$-\mbox{\rm a.s.}
\end{enumerate}
\end{corollary}
This corollary extends to our model the limit theorem for the maxima and the law of the
iterated logarithm of the simple random walk.
\subsection{Subcritical Regime}
\label{subs1}
First, we recall the definition of regularly varying sequences (see for example  \cite{rvs} or
Section 1.9 of \cite{rvbook}).
\begin{definition}
\label{defrv}
Let $r:=(r_n)_{n\ge 1}$ be a sequence of positive reals.
We say that $r$ is {\em regularly varying} with index $\rho\in\rr$, if   $r_n=n^{\rho} \ell_n$, where
$\ell := (\ell_n)_{n\ge 1}$ is such that for any $\lambda>0$,
$\lim_{n\to\infty} \ell_{[\lambda n]}/\ell_n=1$. \\
The set of regularly varying sequences with index $\rho$ is denoted by $\mbox{RV}(\rho)$.
If $r\in  \mbox{RV}(0)$, we say that $r$ is {\em slowly varying}.
\end{definition}
In this section we make the following assumption:
\begin{assume}
[{\bf subcritical regime}]
\label{maina}~\\
Assume that $\veps \in \mbox{RV}(-\alpha)$ for some $\alpha \in [0,1]$. Moreover,
\begin{itemize}
\item
if $\alpha =0$, assume in addition that $\lim_{n\to\infty} \veps_n = 0;$
\item
if $\alpha =1,$ assume in addition that $\lim_{n\to\infty} n\veps_n/ \log n = \infty$.
\end{itemize}
\end{assume}
To state our results for this regime, we need to introduce some
additional notations. We say that two sequences of real numbers
$(x_n)_{n\ge 1}$ and $(y_n)_{n\ge 1}$ are asymptotically equivalent
and write $x_n \sim y_n$ if $\lim_{n\to\infty} x_n / y_n=1$. Let
\beqn \label{tin} T_0=0\quad \mbox{and}\quad T_{n+1}=\inf\{k>
T_n:X_k=0\},~n\ge 0.\feqn That is, $T_n$ is the time of the $n$-th
return to $0$. Let \beqn \label{aen} a_n=n+\sum_{i=1}^n
\fracd{1}{\veps_i},\qquad c_n=\min\{i\in\nn:a_i \geq n\},\quad
\mbox{~and~}\quad b_n=\frac{1}{\veps_{c_{n}}}. \feqn
Lemma~\ref{lem:timoments} below shows that $a_n=E (T_n).$ The
sequence $(c_n)_{n\ge 1}$ is an inverse of $(a_n)_{n\ge 1},$ and, by
a renewal theorem of Smith \cite{smith}, $c_n \sim E (\eta_n)$.
Therefore, $b_n$ can be understood as a typical lifetime of the last
excursion from the origin completed before time $n$. The sequences
$(a_n)_{n\ge 1},(b_n)_{n\ge 1},$ and $(c_n)_{n\ge 1}$ are regularly
varying, and their asymptotic behavior, as $n\to\infty,$ can be
deduced from the standard results collected in Theorem~\ref{th:rv}
(see Corollary~\ref{sequen}). For the distinguished case
$\veps_n=n^{-\alpha}$ with $\alpha\in(0,1),$ we have  $a_n \sim
(1+\alpha)^{-1}n^{1+\alpha},$ $c_n\sim
(1+\alpha)^{\frac{1}{1+\alpha}}n^{\frac{1}{1+\alpha}},$ and hence
$b_n
\sim(1+\alpha)^{\frac{\alpha}{1+\alpha}}n^{\frac{\alpha}{1+\alpha}}.$
\par
We have:
\begin{theorem}
\label{walk}
Let Assumption~\ref{maina} hold. Then, as $n\to\infty$,
$X_n/b_n$ converges in distribution to a random variable with
density $e^{-2|x|}$, $x\in (-\infty,\infty).$
\par
\end{theorem}
Due to the symmetry of the law of $X$, the theorem is equivalent to
the statement that $|X_n|/b_n$ converges in distribution to a
rate-$2$ exponential random variable. The proof of
Theorem~\ref{walk} is based on a comparison of the distribution of
$X_n$ to a stationary distribution of an oscillating random walk
with constant drift $\veps_{c_{n}}$ toward the origin.
\par
We proceed with a more precise description of $X,$ from which
Theorem~\ref{walk} can be in fact derived in an alternative way (see
Remark~\ref{alter} below). Interestingly, the method we use to
establish these more precise results could possibly be adapted to
the non nearest neighbor case, provided one could show in this more
general setting that the number of visits to the origin is
well-localized around its typical value. In this more general case,
the method evoked in Remark~\ref{alter} would also remain valid.
\par
Let $\mn^{(c)}$ denote Ito's excursion measure associated with the
excursions of the Brownian motion with drift $c<0$ above its infimum
process, and let $\zeta$ denote the lifetime of an excursion above
the infimum (see Section~\ref{positives} below for details). Let
\beq V_n=\max\{i\leq n: X_i=0\}. \feq We have:
\begin{theorem}
\label{lastvislastexc}
Let Assumption~\ref{maina} hold. Then:
\begin{itemize}
\item [(i)] $\lim_{n \to \infty} b_{2n} P(X_{2n}=0)=2$.
\item [(ii)]
For $t>0$, $\displaystyle \lim_{n\to \infty} b_{2n}^2 P\bigl(V_{2n}=2n-2[tb_{2n}^2]\bigr)=
2\mn^{(-1)}(\zeta>2t).$
\\
In particular, $\lim_{n\to\infty}P\bigl((2n-V_{2n})/b_{2n}^2\leq x\bigr)=\int_0^{2x}
\mn^{(-1)}(\zeta>t)dt$ for all $x>0.$
\item [(iii)] For $n\in\nn,$ let $Z_n=\bigl(Z_n(t)\bigr)_{t\in \rr_+}$ be a continuous process for which
$Z_n\bigl(k\cdot b_{2n}^{-2}\bigr)=$ $|X_{(V_{2n}+k)\wedge T_1}|\cdot b_{2n}^{-1}$ whenever $k \in \zz_+,$ and which is
linearly interpolated elsewhere.
\\
Then, as $n\to\infty,$ the process $Z_n$ converges weakly in
$C(\rr_+,\rr)$ to a non-negative process with the law $\displaystyle \int_0^{\infty}\mn^{(-1)}(~\cdot~,\zeta>t)dt.$
\end{itemize}
\end{theorem}
Part (i) states that, similarly to the classical renewal theory (cf.
\cite{feller1,feller2}), the probability to find the random walk at
the origin at time $2n$ is asymptotically reciprocal to the expected
duration of the of the last excursion away from the origin completed
before that time. Part (ii) provides limit results on the law of the
last visit time to the origin before a given time. It turns out that
under Assumption \ref{maina}, $b_{2n}^2$ is of smaller order that
$n$ (see Lemma \ref{sequen} below).  In particular, in contrast to
the classical arc-sine law (cf. \cite[p.~196]{durrett}), $V_{2n}/2n$
converges in probability to $1$. Finally, part (iii) is a limit
theorem for the law of excursion away from 0 straddling time $2n.$

\par
The next theorem concerns the asymptotic behavior of the maxima of
$X.$ Let \beqn \label{aihn} h_n :=\fracd{1}{2} b_n \log (c_n/b_n
)=\fracd{\log(\veps_{c_{_{n}}} c_n)}{2\veps_{c_{_{n}}}}. \feqn Note
that by Assumption~\ref{maina}, $\veps_{c_{_{n}}} c_n\to\infty$ as
$n\to\infty.$ Moreover, Corollary~\ref{sequen}-(v) below shows that
\beq \lim_{n\to\infty} \fracd{\log(\veps_{c_{_{n}}} c_n)}{\log n}=
\fracd{1-\alpha}{1+\alpha} .\feq When $\veps_n=n^{-\alpha}$ with
$\alpha \in (0,1),$ we have $h_n\sim
\fracd{1}{2}(1-\alpha)(1+\alpha)^{\frac{-1+\alpha}{1+\alpha}}
n^{\frac{\alpha}{1+\alpha}}\log n$ as $n\to\infty.$ \par Recall the
random variables $M_n$ defined in \eqref{maxima}. We prove in
Section~\ref{sec:sub}:
\begin{theorem}
\label{supt} Let Assumption~\ref{maina} hold. Then \beq
\lim_{n\to\infty} \fracd{1}{\log (\veps_{c_{n}} c_n)} \log\bigl(-
\log P(M_n \leq x h_n)\bigr)=1-x,\qquad x\in
(0,\infty)\backslash\{1\}.\feq In particular, \beq \lim_{n\to\infty}
\fracd{1}{\log (\veps_{c_{n}} c_n)} \log P\bigl(M_n  >x
h_n\bigr)=1-x,\qquad x>1. \feq The above limits remain true when
$M_n$ is replaced with $|M_n|$.
\end{theorem}
\begin{corollary}
\label{cor:supt} Let Assumption~\ref{maina} hold. Then
\beq \limsup_{n\to\infty} X_n / h_n = \lim_{n\to\infty} M_n/h_n =
\lim_{n\to\infty} |M_n|/h_n =1,\feq  where the limits hold $P$-\mbox{\rm a.s.}
when $\alpha<1$ and in probability when $\alpha=1.$
\end{corollary}
We remark that under Assumption~\ref{maina}, $\lim_{n\to\infty} h_n / b_n = \infty,$ and hence
$\lim_{n\to\infty} X_n/M_n=0$ in probability. In particular, Theorem~\ref{walk} cannot be extended
to a functional CLT for a piecewise-linear interpolation of $X_n/b_n$ in $C(\rr_+,\rr).$
\section{Preliminaries}
\label{sec:prel} The goal of this section is threefold. First, in
a series of lemmas we state in Section~\ref{fixed} some general facts about
the measure $P^{\veps}$  in the case when $\veps$ is a constant sequence.
Second, in Section~\ref{vary}, we recall some useful properties of regularly varying
sequences (see Theorem~\ref{th:rv}), and then apply this theorem (see Corollary~\ref{sequen}) to draw conclusions regarding
$a_n,$ $b_n,$ and $c_n$ defined in \eqref{aen}. Finally,
in Section~\ref{positives} we deal with the asymptotic behavior
of a sequence of random walks with a negative drift conditioned to stay positive.
Lemma~\ref{meander} is the key to the proof of the last two parts of Theorem~\ref{lastvislastexc}.
\subsection{Random walks with a negative drift and oscillating random walks}
\label{fixed}
For a real $\delta \in [0,1),$ let $(\delta)$ denote the constant sequence
$\delta,\delta,\ldots$ To simplify the notations we write
$P^{(\delta)}_j$ for $P_{(j,1)}^{(\delta)},$ $P^{(\delta)}$ for
$P_{(0,1)}^{(\delta)}$, and let $E^{(\delta)}_j$ and $E^{(\delta)}$
denote the respective expectation operators. We remark that
$P^{(0)}=\pp$ while $P^{(\delta)}$ with $\delta\in(0,1)$ correspond
to so-called oscillating random walks (cf. \cite{kemperman,
borovkov}). If $\mu$ is a probability distribution on $\zz,$ we
write $P^{(\delta)}_{\mu}$ for the probability measure $\sum_{j\in
\zz}\mu(j) P^{(\delta)}_j$ and let $E^{(\delta)}_{\mu}$ denote the
corresponding expectation.
\par
Recall $T_n$ from \eqref{tin} and set \beqn \label{tau} \tau_n =
T_{n}-T_{n-1},\quad n\geq 1,\feqn where we convene that
$\infty-\infty=\infty.$ That is, $\tau_n$ is the duration of the
$n$-th excursion away from $0$. In the following lemma we recall a
well-known explicit expression for the moment generating function of
$\tau_n$ (see for instance\cite[p.~273]{feller1} or
\cite[p.~276]{durrett}). The moments of $\tau_n$ can be computed
as appropriate derivatives of the generating function.
\begin{lemma}
\label{lem:timoments} Let $\delta \in [0,1)$. Then \beq
E^{(\delta)}\bigl( s^{\tau_1}\bigr)=
\fracd{1-\sqrt{1-(1-\delta^2)s^2}}{1-\delta}\mbox{ \rm  for }0< s <
\fracd{1}{\sqrt{1-\delta^2}}. \feq In particular,
\item $E  \bigl(s^{\tau_n}\bigr)=
\fracd{1-\sqrt{1-(1-\veps_n^2)s^2}}{1-\veps_n}$ for $0< s <
(1-\veps_n^2)^{-1/2}.$
\item  $E(\tau_n)=1+\veps_n^{-1}.$
\item  $E(\tau_n^2)=1+\veps_n^{-1}+\veps_n^{-2}+\veps_n^{-3}.$
\item  $E(\tau_n^3)=1+\veps_n^{-1}+3\veps_n^{-4}+3\veps_n^{-5}.$
\end{lemma}
For our proofs in Sections~\ref{sec:super} and \ref{sec:sub}, we
need the following monotonicity result.
\begin{lemma}
\label{lem:monot} Let $\veps^1:=(\veps^1_n)_{n\ge 1}$ and
$\veps^2:=(\veps^2_n)_{n\ge 1}$ be two sequences such that
$\veps^j_n \in (0,1)$ for $j=1,2$ and $n\in\nn,$ and $\sup_{n\ge 1}
\veps^2_n \le \inf_{n\ge 1} \veps^1_n.$ Further, let $x_1,x_2\in
\zz_+$ be such that $x_2-x_1\in 2\zz_+$. Then there exist two
processes $Y^{j}:=(Y^j_n)_{n \ge 0},~j=1,2$, defined on the same
probability space, such that
\begin{enumerate}
\item For $j=1,2,$ $Y^j$ has the same distribution as $X$ under $P^{\veps^j}_{x_j}$.
\item $|Y^1_n| \le |Y^2_n|$ for all $n\ge 0$.
\end{enumerate}
\end{lemma}
\begin{proof}
Let $(U_n)_{n\ge1}$ be an IID sequence of uniform random variables
on $[0,1]$. For $j=1,2,$ set $Y^1_0=x_1,Y^2_0=x_2,$ $\eta^j_0=1,$
and let \beq Y^{j}_{n+1}=Y^{j}_n+2\Ind_{\bigr\{U_n \ge \frac
12\bigl(1+\mbox{sign}(Y_n^j)\veps^j_{\eta^j_n}\bigr)\bigl\}}-1
\quad \mbox{and} \quad
\eta^{j}_{n+1}=\eta^j_n+\Ind_{\{Y^j_{n+1}=0\}}. \feq Clearly,
$(Y_n^j)_{n\geq 0}$ has the same distribution as $X$ under
$P^{\veps^j}_{x_j}$. Moreover, using induction, it is not hard to
check that for all $n\geq 0,$ $|Y^{2}_{n+1}|-|Y^2_{n}| \ge
|Y^1_{n+1}|-|Y^1_n|$, unless $Y^1_n=0.$ But, since $Y_n^2-Y_n^1$ is
an even integer, $|Y^1_{n+1}|=1\leq |Y^2_{n+1}|$ also in the latter
case.
\end{proof}
In the next lemma, to avoid dealing with a periodic Markov chain,
we focus on the process $(X_{2n})_{n\geq 0}$ rather than on
$X=(X_n)_{n\geq 0}$ itself. It is well-known (see \cite{borovkov} for a
closely related general result) that the law of the Markov chain $X_{2n}$
under $P^{(\delta)}$ converges to its unique stationary
distribution $\mu_\delta.$ The latter is given by
\begin{equation}
\label{eq:invmeas9} \mu_{\delta}(0)=\fracd{2\delta}{1+\delta},
\quad \mu_{\delta}(2i)= \fracd{2\delta (1-\delta)}{(1+\delta)^3}
\Bigl(\frac{1-\delta}{1+\delta}\Bigr)^{2(|i|-1)},~i \in \zz \setminus \{0\}.
\end{equation}
Let $T=\inf\{n\geq 0: X_n=0\}.$ A standard coupling construction for
countable stationary Markov chains (see for instance \cite[p. 315]{durrett}) implies that
\begin{equation}
\label{eq:totvar_dist} \sup_{A\subset 2\zz_+} |
P^{(\delta)}(X_{2n}\in A)-\mu_{\delta}(A)|\le
P^{(\delta)}_{\mu_\delta}(T>2n).
\end{equation}
Estimating the righthand side of \eqref{eq:totvar_dist} we get:
\begin{lemma}
\label{lem:convrate} For all $\delta\in (0,1)$ and $n\ge 1$, \beq
\sup_{A\subset 2\zz_+} | P^{(\delta)}(X_{2n}\in
A)-\mu_{\delta}(A)|\le 2(1+\delta^2)^{-n}. \feq
\end{lemma}
\begin{proof}[Proof of Lemma \ref{lem:convrate}]
By Chebyshev's inequality, for every $\lambda >0,$
\begin{equation}
\label{eq:morecheb}
P_{\mu_\delta}(T>2n) \le e^{-2\lambda n}  E^{(\delta)}_{\mu_{\delta}} ( e^{\lambda T}).
\end{equation}
By Lemma~\ref{lem:timoments}, for $j\in \zz,$
\begin{equation}
\label{eq:eTj}
E^{(\delta)}_j(e^{\lambda T})=\bigl[E^{(\delta)}_1(e^{\lambda T})\bigr]^{|j|}=
\Bigl[\frac{1-\sqrt{1-(1-\delta^2)e^{2\lambda}}}{(1-\delta)e^{\lambda}}\Bigr]^{|j|},~e^{2\lambda}(1-\delta^2)<1.
\end{equation}
Note that the extra term $e^\lambda$ (comparing to the statement of
Lemma~\ref{lem:timoments}) in the denominator corresponds to the
difference between the definition of $\tau_1$, the time of the first {\sl
return} to $0,$ and $T$, the time of the first {\sl visit} to $0.$
\par
Choose $\lambda>0$ such that $e^{2\lambda}=1+\delta^2$. Clearly,
$e^{2\lambda}(1-\delta^2)=(1-\delta^4)<1$. Therefore, \beq
&& P_{\mu_\delta}(T>2n) \underset{\eqref{eq:morecheb}}{\le}
(1+\delta^2)^{-n} \sum_{j \in \zz}
\mu_{\delta}(2j) E^{(\delta)}_{2j} (e^{\lambda T})\\
&&\qquad  \underset{\eqref{eq:invmeas9},\eqref{eq:eTj}}{=}
\fracd{1}{(1+\delta^2)^n}\frac{2\delta}{1+\delta}\Bigl[1+\frac{2(1-\delta)}{(1+\delta)^2}
\sum_{j=1}^{\infty}\Bigl(\frac{1-\delta}{1+\delta} \Bigr)^{2(j-1)}
\Bigl(\frac{1-\sqrt{1-(1-\delta^2)e^{2\lambda}}}{(1-\delta)e^{\lambda}}\Bigr)^{2j}\Bigr]
\\ &&\qquad \quad = \fracd{1}{(1+\delta^2)^n}\frac{2}{1+\delta}
\leq 2(1+\delta^2)^{-n},
\feq
completing the proof.
\end{proof}
\subsection{Regularly varying sequences}
\label{vary}
We next recall some fundamental properties of regularly varying
sequences that are required for our proofs in the subcritical
regime.
\begin{theorem}
\label{th:rv} \cite{rvbook}, \cite{rvs} Let $r:=(r_n)_{n\ge 1}\in
\mbox{RV}(\rho)$ for some $\rho \in \rr $.
\begin{enumerate}
\item  Suppose that  $\rho > -1$. Then
$\lim_{n\to\infty} \fracd{1}{nr_n}\sum_{m=1}^n r_m=\fracd{1}{1+\rho}.$
\item Suppose that $\rho\ge 0$.  Let $(j_n)_{n\ge 1}$ be a sequence of integers such that
$\lim_{n\to\infty} j_n / n = \gamma$ for some $\gamma\in (0,1].$
Then $\max_{j_n\le i \le n } r_i \sim r_n$ and $\min_{ j_n \le i \le
n} r_i \sim \gamma^{\rho} r_n$ as $n\to\infty.$
\item
Suppose that $\rho >0$. Let $r^{\mbox{\rm \tiny inv}}:=(r^{\mbox{\rm
\tiny inv}}_n)_{n\ge 1}$, where $r^{\mbox{\rm \tiny inv}}_n =
\min\{i\ge 1:r_i \geq n\}$. Then $ r^{\mbox{\rm \tiny inv}}_n \in
\mbox{RV}(1/\rho)$ and $r^{\mbox{\rm \tiny inv}}_{[r_n]} \sim
r_{[r^{\mbox{\rm \tiny inv}}_n]}\sim n$ as $n\to\infty.$
\item Suppose that $\rho=0.$ Then $\lim_{n\to\infty} \fracd{\log r_n}{\log n}=0.$
\end{enumerate}
\end{theorem}
\begin{corollary}
\label{sequen} Let Assumption~\ref{maina} hold and 
recall $a=(a_n)_{n\in\nn},$ $b=(b_n)_{n\in\nn},$ and $c=(c_n)_{n\in\nn}$ introduced in \eqref{aen}. We have
\begin{enumerate}
\item   $a_n\sim (1+\alpha)^{-1}n\veps_n^{-1}$ as $n\to\infty.$
In particular, $a\in \mbox{RV}(1+\alpha).$
\item $c\in \mbox{RV}(1/(1+\alpha))$.
\item $b_n=\veps_{c_{_{n}}}^{-1}\sim (1+\alpha) n/c_n$ as $n\to\infty.$
In particular, $b\in \mbox{RV}(\alpha/(1+\alpha))$.
\item $\displaystyle \lim_{n \to \infty}
\frac{n}{b_n^2\log b_n} =\lim_{n\to\infty} \fracd{c_n^2}{n\log
b_n}=\infty.$
\item $\lim_{n\to\infty} \fracd{\log(b_n/c_n)}{\log n}=
\fracd{1-\alpha}{1+\alpha} .$
\end{enumerate}
\end{corollary}
Part (i) of the corollary follows from Theorem \ref{th:rv}-(i). Once
this is established, part (ii) follows from
Theorem~\ref{th:rv}-(iii). Next, claims (i) and (iii) of
Theorem~\ref{th:rv} imply that \beq c_n \veps_{c_{_{n}}}^{-1}\sim
(1+\alpha)a_{c_{_{n}}} \sim (1+\alpha) n,\qquad
\mbox{as}~n\to\infty,\feq which proves (iii). To see that (iv) holds
true observe that part (iii) along with Assumption~\ref{maina}
imply: \beq \fracd{n \veps_{c_{_{n}}}^2}{\log
(\veps^{-1}_{c_{_{n}}})}\sim \fracd{1}{1+\alpha} \cdot \fracd{c_n
\veps_{c_{_{n}}}}{\log (\veps_{c_{n}}^{-1})} \to
\infty~\mbox{as}~n\to \infty. \feq Finally, (v) follows from (ii)
and (iii) combined with Theorem~\ref{th:rv}-(iv).
\subsection{Random walks conditioned to stay positive}
\label{positives}
The aim of this section is to prove Lemma~\ref{meander} below.
We start by recalling some features of the excursion measure of negatively drifted Brownian motion above its infimum
(cf Chapter~VI.8 in \cite{rw2}, in particular Lemma~VI.55.1).
\par
Let $(Z_t)_{t\geq 0}$ be the canonical process
on $C(\rr_+,\rr),$ namely $Z_t(\omega)=\omega(t)$ for $\omega\in C(\rr_+,\rr),$
and, for $c\leq 0,$ let $\bw^{(c)}$ be the law on $C(\rr_+,\rr)$ which makes $Z_t-ct$ into the standard
Brownian motion. For $t\in\rr_+,$ let $Y_t=Z_t-\inf\{Z_s:s\leq t\},$
$\zeta=\inf\{t>0: Y_t=0\},$ and define $\witi Y_t=Y_{t\wedge \zeta}.$ Then $\witi Y=\bigl(\witi Y_t\bigr)_{t\geq 0}$ is a
time-homogeneous continuous Markov process "killed at zero'' with taboo transition density function
\beq
&&
\calp_t^{(c)}(x,y):=\fracd{\bw^{(c)}\bigl(\witi Y_t\in dy,\zeta>t \bigl|\witi Y_0=x\bigr)}{dy}=\\
&& \qquad \qquad =\fracd{1}{\sqrt{2\pi t}}
e^{c(y-x)-c^2t/2}\bigl[e^{-(y-x)^2/2t}-e^{-(y+x)^2/2t}\bigr], \qquad
x,y>0,t>0. \feq In words, $\witi Y$ is an excursion of the Brownian
motion with drift $c\leq 0$ above its infimum process and $\zeta$ is
its lifetime.
\par
For $\omega\in C(\rr_+,\rr)$ let
$\zeta(\omega)=\inf\{t>0:\omega(t)=0\},$ and let \beq U=\{f\in
C(\rr_+,\rr):\omega(0)=0~\mbox{and}~\omega(t)=0~\mbox{for}~t>\zeta(f)\}
\feq be the space of excursions from zero. By Ito's theorem, under
$\bw^{(c)},$ the excursions of the process $Y=(Y_t)_{t\geq 0}$ away
from zero form a Poisson point process on $(0,+\infty)\times U$ with
intensity $dt\times \mn^{(c)}.$ The finite-dimensional distributions
of $\mn^{(c)}$ can be expressed as follows. Let \beq
\calr_t^{(c)}(y):=\fracd{2y}{\sqrt{2\pi t^3}}
\exp\Bigl(-\fracd{(y-ct)^2}{2t}\Bigr),\qquad y>0, t>0. \feq Then,
for $0<t_1<\ldots <t_m$ and $x_1,\ldots,x_m>0,$ \beqn
\label{emeasure} \mn^{(c)}\{f(t_k)\in dx_k:1\leq k\leq
m\}=\calr_{t_1}^{(c)}(x_1)dx_1\prod_{k=2}^m
\calp_{t_{k-1},t_k}^{(c)}(x_{k-1},x_k)dx_k. \feqn The law
$\calr_t^{(c)}(y)dy$ is called the entrance law associated with
$\mn^{(c)}.$ Note that \beqn \label{duration}
\mn^{(c)}(\zeta>t)=\int_0^\infty \calr^{(c)}_t(y)dy. \feqn In
particular, $\mn^{(0)}(\zeta>t)=\sqrt{\fracd{2}{\pi t}}$ whereas for
$c<0,$ $\mn^{(c)}(\zeta>t)=|c|\cdot \mn^{(-1)}(\zeta>tc^2).$ More
generally, \eqref{emeasure} implies that for any constant $c<0,$
\beqn \label{scaling} \mn^{(c)}\bigl(\bigl(|c|\cdot
f(t/c^2)\bigr)_{t\in \rr_+}\in \cdot~\bigr)=|c|\cdot
\mn^{(-1)}\bigl( \bigl(f(t)\bigr)_{t\in \rr_+}\in \cdot~\bigr).
\feqn For $m>0,$ let $C[0,m] : = \{ f : [0,m] \to \rr, f \mbox{
continuous } \}$, equipped with the topology of uniform convergence.
Let $\pi_m:C(\rr_+,\rr)\to C[0,m]$ be the canonical projection
defined by $\pi_m \omega(t)=\omega(t)$ for $t\in [0,m].$ Let
$\mn^{(c)}(~\cdot~|\zeta>t):=\fracd{\mn^{(c)}(~\cdot~;\zeta>t)}{\mn^{(c)}(\zeta>t)}.$
A non-homogeneous in time Markov process $W_+$ on $C[0,1]$ with the
law \beq \mmm^{(0)}(A):=P(W_+\in A~)=\mn^{(0)}(\pi^{-1}_1 A| \zeta
>1), \quad A~\mbox{is a Borel subset of}~C[0,1], \feq is called
Brownian meander (see for instance \cite{bertoin,miller} and
references therein for further background). The meander is a weak
limit of zero-mean random walks conditioned to stay positive (see
\cite{bolthausen,iglehart} and \cite{doney}). Its finite-dimensional
distributions were first computed in \cite{belkin}, and it is not
hard to check that these are consistent with our definition of the
meander. The Brownian meander can also be understood as a Brownian
motion in $C[0,1]$ conditioned to stay positive up to time 1,
defined rigorously with the help of an appropriate $h$-transform.
\par
Analogously, for $c<0,$ we call a non-homogeneous in time Markov process $W^{(c)}_+$ on $C[0,1]$ with the law
\beq
\mmm^{(c)}(A):=P(W^{(c)}_+\in A~)=\mn^{(c)}(\pi^{-1}_1 A| \zeta >1), \quad
A~\mbox{is a Borel subset of}~C[0,1],
\feq
a drifted Brownian meander with drift $c.$
\par
It is well-known a sequence of random walks with well-chosen
asymptotically vanishing drifts converges in distribution to drifted
Brownian motion (see for instance Theorem II.3.2 in \cite{jacod}).
Part (ii) of the following lemma asserts that such walks, when
conditioned to stay positive up to the scaling time, also converge
to a non-degenerate limit, which, not surprisingly, is the drifted
Brownian meander. Part (iii) is then a direct consequence of this
fact.
 Recall the notation $P^{(\delta)}$ was introduced in the first
paragraph of the section and corresponds to a constant sequence
$\delta,\delta,...$. Define \beqn \label{positive}
\Lambda_n=\{X_1>0,\ldots,X_n>0\}. \feqn
\begin{lemma}
\label{meander} Let $(j_n)_{n \in \nn}$ be a sequence of positive reals and $(m_n)_{n \in \nn}$ be
sequence of positive integers such that $\lim_{n \to \infty}j_n =
\infty$, $\lim_{n\to\infty} j_n/j_{n+1}=1,$ and $\lim_{n \to \infty}
\veps_{m_n} j_n = \gamma \in (0, \infty)$.
\par
Then,
\begin{enumerate}
\item $\lim_{n\to\infty} j_n P^{(\veps_{m_n})}(\Lambda_{[j_n^2]})
=\fracd{1}{2} \mn^{(-\gamma)}(\zeta > 1). $
\item
For $n\in\nn,$ let $Y_n=\bigl(Y_n(t)\bigr)_{t\in \rr_+}$ be a continuous process for which
$Y_n\bigl(j_n^{-2}k\bigr)=j_n^{-1}X_k$ whenever $k \in \zz_+,$ and which is
linearly interpolated elsewhere.
\\
Then the process $\pi_1 Y_n$ under $P^{(\veps_{m_n})}(~\cdot~|\Lambda_{[ j_n^2]})$ converges weakly in
$C[0,1]$ to a drifted Brownian meander with drift $-\gamma.$
\item
For $n\in\nn,$ let $\witi Y_n=\bigl(\witi Y_n(t)\bigr)_{t\in \rr_+}$ be a continuous process for which
$\witi Y_n\bigl(j_n^{-2}k\bigr)=$ $j_n^{-1}X_{k \wedge T_1}$ whenever $k \in \zz_+,$ and which is
linearly interpolated elsewhere.
\\
Then the process $\witi Y_n$ under
$P^{(\veps_{m_n})}(~\cdot~|\Lambda_{[ j_n^2]})$ converges weakly in
$C(\rr_+,\rr)$ to a process with law
$\mn^{(-\gamma)}(~\cdot~|\zeta>1)$.
\end{enumerate}
\end{lemma}
\begin{proof}
Since $P^{(\veps)}(\Lambda_{j})$ is a non-increasing function of $j$ and
$j_n/j_{n+1}\sim 1$ as $n\to\infty,$ we can assume without loss of
generality that $[j_n^2]\in 2\zz_+$.
\par
The proof of the lemma is based on the fact that, as we already
mentioned, the result is known for a symmetric random walk, and that
we can explicitly compare the law of a nearest-neighbor drifted walk
and the distribution $\mathbb{P}$ of the simple random walk. Set
$\veps_{m_n}=\delta_n$ and $J_n=\{y\in \rr: yj_n\in\nn\}.$ Counting
the number of steps to the right and to the left, we obtain for any
$m\in \nn,$ $0 <t_1<\ldots <t_m \leq1$ and $y_1,\ldots, y_m \in
\rr,$ $y\in J_n,$ \beqn \nonumber && P^{(\delta_n)}\bigl(
Y_n(t_k)=y_k,~k=1,\ldots,m;X_{[j_n^2]}=2yj_n\bigr) = \\
\label{radon} && \qquad =\pp\bigl( Y_n(t_k)=y_k,~k=1,\ldots,m;
X_{[j_n^2]}=2yj_n\bigr) \fracd{(1-\delta_n^2)^{[j_n^2]/2}
}{1-\delta_n} \Bigr(\frac{1-\delta_n}{1+\delta_n}\Bigr)^{yj_n},
\feqn where the extra factor $(1-\delta_n)^{-1}$ is due to the fact
that the transition kernels of the random walk under
$P^{(\delta_n)}$ and $\pp$ coincide at the origin. In particular,
\beq P^{(\delta_n)}(\Lambda_{[j_n^2]}) = \!\!\sum_{y\in J_n} \!\!
\pp\bigl( \Lambda_{[j_n^2]}, X_{[j_n^2]}=2y j_n\bigr)
\fracd{(1-\delta_n^2)^{[j_n^2]/2} }{1-\delta_n}
\Bigr(\frac{1-\delta_n^2}{1+\delta_n}\Bigr)^{yj_n}.\feq $\mbox{}$
\\ (i)~Using the identity $\pp(
\Lambda_{[j_n^2]}) = \fracd{1}{2} \pp(X_{[j_n^2]}=0)$ (see for instance \cite[p. 198]{durrett}), we obtain:
\beq
P^{(\delta_n)}(\Lambda_{[j_n^2]})&=&
\fracd{j_n}{2}\int_0^{\infty} du\pp
(X_{[j_n^2]}=0) \pp\bigl(X_{[j_n^2]}=2[j_n u]
\bigl| \Lambda_{[j_n^2]}\bigr) \fracd{(1-\delta_n^2)^{[j_n^2]/2}}{1-\delta_n}
\Bigl(\fracd{1-\delta_n}{1+\delta_n}\Bigr)^{[j_n u]}.
\feq
The local limit theorem for the simple random walk (see for instance \cite[p. 199]{durrett}) implies
that
\beqn
\label{local}
 \lim_{n\to\infty} j_n \pp (X_{[j_n^2]}=0) = 2\lim_{n\to\infty} j_n \pp (\Lambda_{[j_n^2]}) =
\sqrt{\fracd{2}{\pi}}. \feqn Furthermore (see for instance
\cite{iglehart}), the sequence of probability measures $(\nu_n)$
defined on Borel sets $A\subset \rr_+$ by \beq \nu_n(A):=j_n\int_A
du\pp\bigl(X_{[j_n^2]}=2[j_n u] \bigl| \Lambda_{[j_n^2]}\bigr) \feq
converges weakly to the Rayleigh distribution on $\rr_+$ with the
density $ue^{-\frac{u^2}{2}}du.$ Using the dominated convergence
theorem, we conclude that \beq  \lim_{n\to\infty} j_n
P^{(\delta_n)}(\Lambda_{[j_n^2]})=\int_0^{\infty} du \
\frac{u}{\sqrt{2\pi }} \exp \Bigl(-\fracd{u^2}{2} - u \gamma -
\frac{\gamma^2 }{2} \Bigr), \feq which proves
Lemma~\ref{meander}-(i) in view of \eqref{duration}.
\\
$\mbox{}$
\\
(ii)~ First we will prove the convergence of finite-dimensional
distributions. It follows from \eqref{radon} that for any $m\in\nn,$ positive reals
$0<t_1<\cdots <t_m\leq 1,$  and Borel sets $A_k\subset \rr_+,$ $k=1,\ldots, m,$ \beqn \label{tight}
&& \qquad \qquad \qquad \qquad
P^{(\delta_n)}\Bigl( Y_n(t_k)\in A_k,~k=1,\ldots,m\bigr| \Lambda_{[j_n^2]}\Bigr) = \\
\nonumber
&& \fracd{\sum_{y\in J_n} \pp\Bigl(Y_n(t_k)\in A_k,~k=1,\ldots,m;X_{[j_n^2]}=2yj_n\bigl|
\Lambda_{[j_n^2]}\Bigr) \pp\bigl( \Lambda_{[j_n^2]}\bigr)
\fracd{(1-\delta_n^2)^{[j_n^2]/2}}{1-\delta_n}\Bigr(\frac{1-\delta_n}{1+\delta_n}\Bigr)^{yj_n}}
{P^{(\delta_n)}\bigl(\Lambda_{[j_n^2]}\bigr)}.
\feqn
Therefore, by the central limit theorem for
random walks conditioned to stay positive (see \cite{bolthausen, iglehart})
combined with the first part of the lemma and \eqref{local},
\beq
&&
\lim_{n\to\infty} P^{(\delta_n)}
\Bigl(Y_n(t_k)\in A_k,~k=1,\ldots,m\bigr| \Lambda_{[j_n^2]}\Bigr)
\\
&& \qquad =\sqrt{\fracd{2}{\pi}}\fracd{1}{\mn^{-\gamma}(\zeta
>1)}\int_0^\infty du \ \mmm^{(0)}(Y_{t_k}\in A_k,~k=1,\ldots,m;Y_1\in
du)
\exp\Bigr(-u\gamma-\fracd{\gamma^2}{2}\Bigr)\\
&& \qquad =\fracd{1}{\mn^{(-\gamma)}(\zeta >1)}\int_0^\infty du \
\mn^{(0)}(Y_{t_k}\in A_k,~k=1,\ldots,m;Y_1\in du)
\exp\Bigr(-u\gamma-\fracd{\gamma^2}{2}\Bigr)
\\
&& \qquad =\mmm^{(-\gamma)}(Y_{t_k}\in A_k,~k=1,\ldots,m). \feq
Next, tightness of the family of discrete distributions follows from
the corresponding result for the simple random walk available in
Section~3 of \cite{iglehart}, along with \eqref{tight}. This
completes the proof of Lemma~\ref{meander}-(ii).
\\
$\mbox{}$
\\
(iii)~ We use the second part of the lemma, along with the fact that
the process $(Y_n(t))_{t \ge 1}$ converges weakly in $C(\rr_+,\rr)$
to a Brownian motion with drift $-\gamma$ (see for instance
\cite[Theorem~II.3.2]{jacod}). The claim then follows immediately from
the Markov property (applied at time $t=1$) under
$\mn^{(-\gamma)}(~\cdot~|\zeta>1)$ (cf. \cite[Section~VI.48]{rw2}).
\end{proof}
\section{Supercritical Regime}
\label{sec:super} This section is devoted to the proof of Theorem
\ref{th:main} and is correspondingly divided into two parts. The
proof of the invariance principle for $X_n $ given in
Section~\ref{part1} uses a decomposition representing $X_n$ as a sum
of a martingale and a drift term. It is then shown that the drift
term is asymptotically small compared to the martingale, and that
the martingale satisfies the invariance principle. The criterion for
the equivalence of $P$ and $\pp$ is proved in Section~\ref{part} by
a reduction to a similar question for the law of the sequence of
independent variables $\tau_n$ defined in \eqref{tau}.
\subsection{Invariance principle for $X_n$}
\label{part1} The first part of the following proposition states
that $T_n/n^2$ converges in distribution, as $n\to\infty,$ to the
hitting time of level $1$ of the standard Brownian motion, a
non-degenerate stable random variable of index $1/2$. The second
part is required to evaluate both the variance of the martingale
term as well as the magnitude of the drift in decomposition
\eqref{eq:tauest2} below.
\begin{proposition}
\label{supera} Assume that $\lim_{n\to\infty} n \veps_n =0$. Then
\begin{enumerate}
\item For $\lambda \ge 0$, $\lim_{n\to\infty} E\bigl(e^{-\lambda
T_n/n^2}\bigr)=e^{-\sqrt{2\lambda}}$.
\item $\fracd{1}{n}\sum_{i=1}^n \veps_i \tau_i$
converges to zero in probability as $n\to\infty$.
\end{enumerate}
\end{proposition}
\begin{proof}~
\\
(i)~It is well-known (see for instance \cite[p. 394]{durrett}) that
\beq \label{eq:bmtn} \lim_{n\to\infty} \ee\bigl( e^{-\lambda
T_n/n^2}\bigr)=\lim_{n\to\infty} \bigl( \ee \bigl( e^{-\lambda
\tau_1/n^2}\bigr)\bigr)^n = e^{-\sqrt{2\lambda}},~\lambda \ge 0.
\feq By Lemma~\ref{lem:monot}-(i), $ E \bigl(e^{-\lambda
T_n/n^2}\bigr) \ge \ee\bigl( e^{-\lambda T_n/n^2}\bigr)$. Hence $
\liminf_{n\to\infty} E \bigl( e^{-\lambda T_n/n^2}\bigr) \ge
e^{-\sqrt{2\lambda}}$. It remains to show that $\displaystyle
\limsup_{n\to\infty} E\bigl( e^{-\lambda T_n/n^2}\bigr)\le
e^{-\sqrt{2\lambda}}$.
\par
Let $\delta\in (0,1)$. Clearly,
\begin{equation}
\label{eq:deltastep1}
E \bigl( e^{-\lambda T_n/n^2} \bigr)\le \prod_{k=[\delta n]}^n E\bigl( e^{-\lambda \tau_k/n^2}\bigr) .
\end{equation}
Thanks to Assumption 2.4, we can take $n$ large enough so that
$k\veps_k \le \delta^2/2$ for all $k\geq [\delta n].$ Then, for
$k\ge [\delta n],$ \beqn \label{esti} \veps_k \le
\frac{\delta^2}{2k}\le \frac{\delta^2}{2[\delta
n]}<\frac{\delta}{n}.\feqn Using Lemma~\ref{lem:monot} to estimate
the product in the righthand side of \eqref{eq:deltastep1}, we get
\begin{equation}
\label{eq:deltastep2} E\bigl( e^{-\lambda T_n/n^2}\bigr) \le \Bigl( E^{(\delta/n)}
\bigl( e^{-\lambda\tau_1/n^2 })\Bigr) ^{(1-\delta)n}.
\end{equation}
Next, we observe that, using Lemma~\ref{lem:timoments},
\begin{align}
\nonumber
E^{(\delta/n)}\bigl(  e^{-\lambda\tau_1/n^2 }\bigr) &
\underset{\mbox{  }}{=}
\frac {1-\sqrt{1-(1-\delta^2/n^2)e^{-2\lambda/n^2}}} {1-\delta/n}
\leq \frac{1-\sqrt{1-e^{-(\delta^2+2\lambda)/n^2}}}{1-\delta/n}\\
\label{eq:tauest}
&=
\ee\bigl( e^{-(\delta^2/2+\lambda) \tau_1/n^2}\bigr) (1-\delta/n)^{-1}.
\end{align}
Hence,
\begin{align*}
\limsup_{n\to\infty} E \bigl(e^{-\lambda T_n/n^2}\bigr) &\underset{
\eqref{eq:deltastep2},~\eqref{eq:tauest}}{\le} \limsup_{n\to\infty}
\bigl(\ee\bigl( e^{-(\delta^2/2+\lambda)
\tau_1 /n^2}\bigr)\bigr)^{[(1-\delta)n]}  (1-\delta/n)^{-(1-\delta)n} \\
&\underset{\eqref{eq:bmtn}}{=}   e^{- (1-\delta) \sqrt{\delta^2 + 2\lambda} } e^{\delta (1-\delta)}.
\end{align*}
Letting $\delta \to 0$ completes the proof of
Proposition~\ref{supera}-(i).
\\
$\mbox{}$
\\
(ii)~Fix $\delta \in (0,1)$ and let $S_1 = \frac 1n \sum_{k=1}^{[\delta
n]-1}\veps_k \tau_k,$ $S_2=\frac 1n \sum_{k=[\delta n]}^n \veps_k
\tau_k$. As before, we assume that $n$ is large enough, so that
\eqref{esti} holds true for all $k\geq [\delta n]$. In particular,
$S_2 \le {\delta} T_n/n^2$. Next,
$$ P( S_1+S_2 \ge 2\sqrt{\delta}) \le P(S_1 \ge \sqrt{\delta})+
P(S_2 \ge \sqrt{\delta}) \le \delta^{-1/2} E (S_1) + P(T_n/n^2 \ge
\delta^{-1/2}).$$ By Lemma~\ref{lem:timoments}, $E (S_1) \le \frac
1n \sum_{k=1}^{[\delta n]}(1+\veps_k)\le 2 \delta$. Therefore,
$$ P(S_1 +S_2 \ge 2\sqrt{\delta})\le 2\sqrt{\delta}+P(T_n/n^2 \ge \delta^{-1/2}).$$
By part (i), the second term goes to $0$ as
$n\to\infty.$ Letting $\delta$ go to $0$ finishes the proof.
\end{proof}
We are now in position to give the proof of the first part of
Theorem~\ref{th:main}.
\begin{proof}[Proof of Theorem \ref{th:main}-(i)]
Recall $\calf_n=\sigma(X_1,\ldots,X_n),$ $d_n=\text{\rm
sign}(X_n)\veps_{\eta_n},$ and identity~\eqref{din_dyn}. Let:
\begin{equation}
\label{eq:tauest2} H_n = X_n - \ol{D}_n\quad \mbox{with}\quad
\ol{D}_n :=\sum_{k=0}^{n-1} d_k.
\end{equation}
It follows from \eqref{din_dyn} that $H:=(H_n,\calf_n)_{n\ge 0}$ is
a martingale.
\par Let $S_n = \sum_{k=1}^{\eta_n} \veps_k \tau_k$. We next prove
the following estimate:
\begin{equation}
\label{eq:Dnsmall} \lim_{n\to\infty} S_n /\sqrt{n} =0,~\text{in
probability.}
\end{equation}
Let $\delta>0$ and $m>0$. Then,
$$ \{S_n >\delta \sqrt{n}\}\subseteq \{\eta_n\ge [\sqrt{m n}] \}\cup
\{\sum_{k=1}^{[\sqrt{mn}]} \veps_k \tau_k>\delta \sqrt{n} \}.$$
Hence, by Proposition \ref{supera}-(ii), $\limsup_{n\to\infty} P(S_n
\ge \delta \sqrt{n}) \le \limsup_{n\to\infty} P(\eta_n \ge
[\sqrt{mn}])$. However, $\{\eta_n\ge [ \sqrt{mn}]\}=
\{T_{[\sqrt{mn}]}\le n\}.$ Therefore, $$\limsup_{n\to\infty} P(S_n
\ge \delta \sqrt{n}) \le \limsup_{k\to\infty} P(T_k/k^2 \leq 2/m).$$
By letting $m\to\infty$, and since $\delta$ is arbitrary,
\eqref{eq:Dnsmall} follows Proposition~\ref{supera}-(i).\par We next
apply the martingale central limit theorem \cite[pp. 412]{durrett}
to show that $H$ satisfies the invariance principle. Let \beq
V_n=\sum_{k=1}^{n} E\bigl( (H_{k+1}-H_k)^2\bigl|\calf_k\bigr)
=\sum_{k=1}^{n} E\bigl( (X_{k+1}-X_k-d_k)^2\bigl| \calf_k\bigr),
\feq Due to the fact that $H$ has bounded increments, it is enough
to verify that $\lim_{n\to\infty} V_n/n=1$ in probability. Note that
by \eqref{din_dyn} $$V_n= \sum_{k=1}^{n} \bigl( 1 - 2d_k^2 +
d_k^2\bigr)=n-\sum_{k=1}^{n} d_k^2,$$ and $\sum_{k=1}^n d_k^2 \le
\sum_{k=1}^n |d_k|\le S_n.$ It follows from \eqref{eq:Dnsmall} that
$\lim_{n\to\infty} \sum_{k=1}^n d_k^2/n=0$ in probability, and, consequently,
the invariance principle holds for $H.$
\par
In order to complete the proof, by \cite[Theorem 2.1, p.11]{bill},
it suffices to show that for all $m>0$ and any continuous function
$\varphi:C[0,m]\to \rr$, we have $\lim_{n\to\infty} E\bigl(
\varphi(\cali^X_n) \bigr)=$ $\lim_{n\to\infty} E\bigl(
\varphi(\cali^{H,m}_n)\bigr),$ where $\cali^{H,m}_n(t)$ coincides
with $\cali^H_n(t)$ on $[0,m].$ Note that the limit in the righthand
side exists due to the invariance principle for $H$. Since $\varphi$
is bounded, uniformly continuous, this will follow once we prove
that
$$K_n :=\max_{t\in [0,m]}
\bigl|\cali^X_n(t)-\cali^{H,m}_n(t)\bigr|\underset{n\to\infty}{\to}0,~\mbox{
in }P\mbox{-probability}.$$ By its definition in \eqref{esn},
$\cali^X_n(t)$ (resp. $\cali^{H,m}_n(t)$) is a convex combination of
$X_{[nt]}$ and $X_{[nt]+1}$ (resp. $H_{[nt]}$ and $H_{[nt]+1}$).
Since $|X_{[nt]}-X_{[nt]+1}|=1$ and $|H_{[nt]}-H_{[nt]+1}|\le 2$, it
follows that $$K_n\leq \max_{t\in[0,m]} \frac{
|X_{[nt]}-H_{[nt]}|+3}{\sqrt{n}}\le \max_{t\in [0,m]}
\frac{S_{[nt]}+3}{\sqrt{n}}\le
\frac{S_{nm}+3}{\sqrt{n}}\underset{n\to\infty}{\to}0\mbox{ in
}P\mbox{-probability},$$ where the limit in the righthand side is
due to \eqref{eq:Dnsmall}.
\end{proof}
\subsection{Criterion for the equivalence of $P$ and $\pp$}
\label{part}
\begin{proof}[Proof of Theorem~\ref{th:main}-(ii)]
Recall $\calf_n=\sigma(X_0,\dots,X_n)$ and let
$\calg_n=\calf_{T_n},$ the $\sigma$-algebra generated by the paths
of $X$ up to time $T_n$. Let $\calf = \sigma(\cup_{n\ge 0} \calf_n)$
and $\calg = \sigma(\cup_{n\ge 0} \calg_n).$
\par
Under both $P$ and $\pp,$ $\lim_{n\to\infty}T_n=\infty$ with
probability one and hence $\calg = \calf$ up to null-measure sets.
Therefore, the measures $P$ and $\pp$ are equivalent if $P|_{\calg}$
and $\pp|_{\calg}$, their restrictions to $\calg,$ are equivalent.
\par
Let $\gamma:=(\gamma_0,\gamma_1,\dots)$ be a random walk path
starting from the origin. That is, $\gamma_0=0$ and
$|\gamma_{n+1}-\gamma_n|=1$ for all $n$. Let $T_0(\gamma)=0$ and,
for $n\geq 1,$ \beqn \label{tigamma}
T_n(\gamma)=\min\{i>T_{n-1}(\gamma): X_i=0\}\quad \mbox{and}\quad
\tau_n(\gamma)=T_n(\gamma)-T_{n-1}(\gamma).\feqn Counting the number
of the steps to the left and to the right during each excursion of
the random walk from zero, we obtain
\begin{align}
\nonumber
P(X_k=\gamma_k,~\forall k\le T_n)&= \prod_{k=1}^{n}
\frac 12 \Bigl(\frac 12 (1+\veps_k)\Bigr)^{\tau_k(\gamma)/2}
\Bigl(\frac 12 (1-\veps_k)\Bigr)^{\tau_k(\gamma)/2-1}\\
\label{counting}
&=2^{-T_n(\gamma)} \prod_{k=1}^n\frac{
(1-\veps_k^2)^{\tau_k(\gamma)/2}}{1-\veps_k},
\end{align}
where the difference between the powers in the righthand side of the
first line is due to the fact that from $0$, the probability of
going either to the right or to the left is $\frac 12$. On the other
hand, $\pp(X_k=\gamma_k,~\forall k\le T_n)=2^{-T_n(\gamma)}$.
\par
For $n\geq 1,$ let \beqn \label{rn-tin}
F_n(\gamma):=\frac{P(X_k=\gamma_k,~\forall k\le
n)}{\pp(X_k=\gamma_k,~\forall k\le n)}=\prod_{k=1}^{n} \frac
{(1-\veps_k^2)^{\tau_k(\gamma)/2}}{1-\veps_k},\feqn and set
$F_{\infty} = \limsup_{n\to\infty} F_n.$ Note that $F_n\in \calg_n$
and hence $F_{\infty}\in\calg$. By \cite[Theorem~3.3,
p.~242]{durrett},
\begin{align*}
&P|_{\calg}\sim\pp|_{\calg}\mbox{ if and only if
}F_{\infty}<\infty,~ P|_{\calg}\mbox{-almost
surely};\\&P|_{\calg}\perp \pp|_{\calg}\text{ if and only if
}F_{\infty}=\infty, ~P|_{\calg}\mbox{-almost surely}.
\end{align*}
Identity \eqref{rn-tin} with $n=1$ shows that distribution of
$\tau_k$ under $P$ is absolutely continuous with respect to its
distribution under $\pp$, and the corresponding Radon-Nikodym
derivative is $(1-\veps_k)^{-1}(1-\veps_k^2)^{\tau_k/2}$. Since
$(\tau_k)_{k\ge 1}$ is a sequence of independent random variables
under both measures, Kakutani's dichotomy theorem (see
\cite[p.~244]{durrett}) implies that
\begin{align*}
& F_{\infty}<\infty\mbox{ or }=\infty,~ P|_{\calg}-\mbox{a.s.},
~\mbox{according to whether}~\lim_{n\to\infty} \ee\bigl(
\sqrt{F_n}\bigr)>0~\mbox{or}~= 0.
\end{align*}
We have:
\begin{align*}
\ee\bigl( \sqrt{F_n}\bigr)& =  \prod_{k=1}^n \ee\Bigl(
\frac{(1-\veps_k^2)^{\tau_k/4}}{\sqrt{1-\veps_k}} \Bigr)
\underset{\mbox{Lemma } \ref{lem:timoments}}{=} \prod_{k=1}^n
\frac{1-\sqrt{1-(1-\veps_k^2)^{1/2}}}{\sqrt{1-\veps_k}}.
\end{align*}
Choose any $\delta\in (0,\sqrt{1/2}-1/2).$ Since $\lim_{k\to\infty} \veps_k=0,$ we have for all $k$ large enough,
\beq
1-\veps_k\sqrt{1/2+\delta}\leq 1-\sqrt{1-(1-\veps_k^2)^{1/2}} \leq 1-\veps_k \sqrt{1/2},
\feq
and
\beq
1-(1/2+\delta)\veps_k\leq \sqrt{1-\veps_k}\leq 1-\veps_k/2.
\feq
In particular, $\lim_{n\to\infty} \ee\bigl( \sqrt{F_n}\bigr)>0$ if
and only if $\sum_{k=1}^{\infty} \veps_k <\infty$.
\end{proof}
\section{Subcritical Regime}
\label{sec:sub}
The goal of this section is to prove the results presented in
Section~\ref{subs1}. In Section~\ref{sec:tnsub} we obtain auxiliary
limit theorems and large deviations estimates for $\eta_n,$ the
occupation time at the origin.
We first prove corresponding results for $T_n,$ and then use the correspondence
between $(T_n)_{n\geq 1}$ and $(\eta_n)_{n\geq 1}.$
Section~\ref{subs:walk} contains the proof of the limit theorem for $X_n$
stated in Theorem~\ref{walk}. In Section~\ref{last} we prove the more refined result
given by Theorem~\ref{lastvislastexc}. Finally, Theorem~\ref{supt} and
Corollary~\ref{cor:supt}, describing the asymptotic behavior of the range
of the random walk, are proved in Section~\ref{supr}.
\subsection{Limit theorems and large deviations estimates
for $T_n$ and $\eta_n$} \label{sec:tnsub}
Let $N(0,\sigma^2)$ denote a normal random variable
with zero mean and variance $\sigma^2.$ We write $X_n\Rightarrow Y$ when
a sequence of random variables $(X_n)_{n\geq 1}$ converges to random variable
$Y$ in distribution. Let \beqn
\label{gin} g_n:=\sqrt{E(T_n^2)-(E(T_n))^2}=\Bigr[{\sum}_{i=1}^n
(\veps_i^{-3}- \veps_i^{-1}) \Bigl]^{1/2},
\feqn
where the first equality is the definition of $g_n$ while the second one follows from
Lemma~\ref{lem:timoments}.
\par
First, we prove the following limit theorem for the sequence $(T_n)_{n\geq 1}.$
\begin{proposition}
\label{tsuper} Let Assumption~\ref{maina} hold.
 Then
\beq
\frac {T_n -a_n}{g_n} \Rightarrow N(0,1),\mbox{ as }n\to\infty.
\feq
In particular, $\lim_{n\to\infty} T_n /a_n =1,$ where the convergence is in probability.
\end{proposition}
Next, we derive from this proposition the following limit
theorem for $(\eta_n)_{n\ge 1}.$
\begin{proposition}
\label{cor} Let Assumption~\ref{maina} hold. Then \beq \frac{\eta_n
- c_n}{\sqrt{n} }\Rightarrow N\Bigl(0,
\frac{1+\alpha}{1+3\alpha}\Bigr).
\feq
In particular, $\lim_{n\to\infty} \eta_n /c_n =1,$ where the convergence is in probability.
\end{proposition}
Finally, we complement the above limit results by the following large deviation estimates.
\begin{proposition}
\label{ldp} Let Assumption~\ref{maina} hold. Then, for $x>0$, \beq
\lim_{n\to\infty} \frac{1}{n\veps_n} \log P\Bigl(
\Bigl|\frac{T_n}{a_n}-1\Bigr|>x\Bigr)<0. \feq
\end{proposition}
\begin{corollary}
\label{ldpeta} Let Assumption~\ref{maina} hold. Then, for $x > 0$,
\beq \lim_{n\to\infty} \frac{b_n^2}{n} \log
P\Bigl(\Bigl|\frac{\eta_n}{c_n}-1\Bigr|>x\Bigr)<0. \feq
\end{corollary}
In both the corollaries above, the existence of the limit is a part of the claim.
\begin{corollary}
\label{weakdelta} Let Assumption~\ref{maina} hold. Then, there
exists a sequence $(\theta_n)_{n\geq 1}$ such that
$\theta_n\in(0,1)$ for all $n,$ $\lim_{n\to\infty} \theta_n=0$ and
\beq \lim_{n\to\infty} \exp\Bigl(\fracd{n}{b_n^2\log n}\Bigr)\cdot
P\Bigl(\Bigl|\frac{\eta_n}{c_n}-1\Bigr|>\theta_n\Bigr)=
\lim_{n\to\infty} b_n^2
P\Bigl(\Bigl|\frac{\eta_n}{c_n}-1\Bigr|>\theta_n\Bigr)=0. \feq
\end{corollary}
We remark that the estimates stated in Corollary~\ref{weakdelta} are
not optimal and, furthermore, the second is actually implied by the
first one. However, the statement in the form given above is
particularly convenient for reference in the sequel.
\par
Corollary~\ref{ldpeta} is deduced from Proposition~\ref{ldp} using a
routine argument similar to the derivation of Proposition~\ref{cor}
from Proposition~\ref{tsuper}, and thus its proof will be omitted.
In turn, Corollary~\ref{weakdelta} is an immediate consequence of
Corollary~\ref{ldpeta} and Corollary~\ref{sequen}-(iv). Indeed,
these two results combined together imply that \beq
\lim_{n\to\infty} \exp\Bigl(\fracd{n}{b_n^2\log n}\Bigr)\cdot
P\Bigl(\Bigl|\frac{\eta_n}{c_n}-1\Bigr|>x\Bigr) =\lim_{n\to\infty}
b_n^2 P(|\eta_n/c_n -1|> x)=0 \feq for all $x>0.$ Let $n_0=1,$ for
$p\in \nn$ let $n_p$ be the smallest integer greater than $n_{p-1}$
such that $\exp\bigl(\frac{n}{b_n^2\log n}\bigr)\cdot
P\bigl(\bigl|\frac{\eta_n}{c_n}-1\bigr|>1/p\bigr)<1/p\quad\mbox{for
all}~n\geq n_p,$ and set $\theta_n=1/p$ for
$n=n_p,\ldots,n_{p+1}-1.$
\par
\begin{proof}[Proof of Proposition \ref{tsuper}]
Let $S_n = (T_n -a_n) /g_n$. By Lemma~\ref{lem:timoments},
$E(S_n)=0$ and $E(S_n^2)=$. By  Lyapunov's version of the
CLT for the partial sums of independent random variables, \cite[p.
121]{durrett}, ${S_n}  \Rightarrow N(0,1)$ if
\begin{equation*}
\label{eq:zero}
\lim_{n\to\infty} \frac{1}{g_n^3} \sum_{m=1}^nE \bigl( |\tau_m-1-E (\tau_m)|^{3}\bigr) =0.
\end{equation*}
By Lemma~\ref{lem:timoments}, and using the fact that $\veps_m\in (0,1),$
\beq
E\bigl(|\tau_m-1-\veps_m^{-1}|^3\bigr) \leq
4E\bigl((\tau_m-1)^3+\veps_m^{-3}\bigr) \leq 4 (8\veps_m^{-5}+\veps_m^{-3}) \leq 36 \veps_m^{-5}.
\feq
Next, by Theorem~\ref{th:rv}-(i),  as $n\to\infty,$ $\sum_{m=1}^n \veps_m^{-5} \sim (1+5\alpha)^{-1} n\veps_n^{-5}$ and
\begin{equation}
\label{eq:gn}
g_n^2  \sim (1+3\alpha)^{-1} n \veps_n^{-3}.
\end{equation}
Therefore,
$$ \fracd{1}{g_n^3}\sum_{m=1}^n \veps_m^{-5} \sim \frac{(1+5\alpha)^{-1} n
\veps_n^{-5}}{(1+3\alpha)^{-3/2} n^{3/2} \veps_n^{-9/2}}=\frac
{(1+3\alpha)^{3/2}}{(1+5\alpha)} \frac{1}{\sqrt{n \veps_n}}\to
0,~\mbox{ as }n\to\infty,$$ where we use Assumption \ref{maina} to
obtain the last limit. This completes the proof of the weak convergence
of $(T_n-a_n)/g_n.$
\par
The convergence of $T_n/a_n$ in probability will follow, provided
that $\lim_{n\to\infty} a_n/g_n =\infty$.
Using again Theorem \ref{th:rv}-(i), and then Assumption~\ref{maina}, we obtain, as $n\to\infty,$
$$ \frac{a_n}{g_n} \sim
\frac{(1+\alpha)^{-1}n\veps_n^{-1}}{(1+3\alpha)^{-1/2} n^{1/2}
\veps_n^{-3/2}}\sim\frac{(1+3\alpha)^{1/2}}{1+\alpha}
\sqrt{n\veps_n}\to \infty,~\mbox{ as }n\to\infty.$$
The proof of the proposition is completed.
\end{proof}
\begin{proof}[Proof of Proposition \ref{cor}]
First, we observe that the second statement of the proposition follows
from the first one and the fact that $\lim_{n\to\infty} c_n /\sqrt{n}=\infty$
(cf. Corollary~\ref{sequen}-(iv)).
\par
We next prove the central limit theorem for $\eta_n.$ As in Proposition \ref{tsuper},
let $g_m$ denote the variance of $T_m$ and let $\witi T_m=(T_m -a_m)/g_m$.
Fix $x\in\rr$. By Corollary~\ref{sequen}-(iv), $x\sqrt{n}+c_n\sim c_n$  as $n\to\infty,$
and hence
\begin{align}
\nonumber
P\Bigl(\frac{\eta_n- c_n}{\sqrt{n}}\le x \Bigr)&= P(\eta_n \le c_n + x\sqrt{n})  =
P(T_{[c_n+x\sqrt{n}]+1}>n)\\
\label{eq:etatot}
& = 1-P\Bigl( \witi T_{[c_n+x\sqrt n]+1} \leq
\fracd{n-a_{[c_n+x\sqrt n]+1}} {g_{[c_n+x\sqrt n]+1}}\Bigr).
\end{align}
By  \eqref{eq:gn}, as $n\to\infty,$ $g_{[c_n+x\sqrt n]+1}  \sim
(1+3\alpha)^{-1/2} (c_n+x\sqrt n)^{1/2}
\veps_{[c_n+x\sqrt{n}]}^{-3/2}\sim \sqrt{ \frac{c_n
\veps_{c_{n}}^{-3}}{1+3\alpha}},$ and hence \beq
\fracd{a_{[c_n+x\sqrt n]+1}-n}{g_{[c_n+x\sqrt n]+1}} \sim
\frac{\sum_{i=c_n}^{[c_n+x\sqrt n]+1}
\veps_i^{-1}}{\sqrt{c_n\veps_{c_{n}}^{-3}/(1+3\alpha)}}
\underset{\mbox{Theorem }\ref{th:rv}-(ii)}{\sim} \frac{ x\sqrt
n\cdot \veps_{c_{n}}^{-1}\sqrt{1+3\alpha}}
{\sqrt{c_n\veps_{c_{n}}^{-3}}}. \feq The rightmost expression above
tends to $x\sqrt{\frac{(1+3\alpha)}{1+\alpha}},$
as $n \to \infty$. Therefore, \beq \lim_{m\to\infty} P\Bigl(\witi T_m \le
x \sqrt{\frac{1+3\alpha}{1+\alpha}}\Bigr)&
\underset{\mbox{Proposition }\ref{tsuper}}{=}& \lim_{m\to\infty} 1-
P\Bigl(\witi T_m \le - x \sqrt{\frac{1+3\alpha}{1+\alpha}}\Bigr)
\\ &\underset{\eqref{eq:etatot}}{=}&\lim_{n\to\infty}
P\Bigl(\frac{\eta_n-c_n}{\sqrt{n}}\le x \Bigr),\feq completing the
proof of Proposition~\ref{cor}.
\end{proof}
\begin{proof}[Proof of Proposition~\ref{ldp}]
Let $\rho_n = \min_{k\le n} \veps_k$. Theorem~\ref{th:rv}-(ii) implies that $\rho_n \sim
\veps_n$  as $n\to\infty.$ Let $\lambda \in (-\infty,\frac 12)$ and define $ \Lambda(\lambda)= \int_0^1
\bigl(x^{-\alpha}-\sqrt{x^{-2\alpha}-2\lambda}\bigr) dx.$ We
shall prove that \beqn \label{eq:expmom} \lim_{n\to\infty}
\frac{1}{n\veps_n} \log E \bigl(e^{\lambda \rho_n^2  T_n} \bigr)=
\Lambda(\lambda). \feqn Once this result is established, we will
deduce the proposition by applying standard Chebyshev's bounds for
the tail probabilities of $T_n.$
\par
To prove \eqref{eq:expmom} we first observe that, by Lemma~\ref{lem:timoments},
\beqn
\label{eq:display}
&&
\frac{1}{n\veps_n} \log E\bigl(e^{\lambda \rho_n^2 T_n}\bigr)=
\frac{1}{n\veps_n}\sum_{i=1}^n
\log\Bigl(1+\fracd{\veps_i-\sqrt{1-(1-\veps_i^2)e^{2\rho_n^2 \lambda}}}{1-\veps_i}\Bigr).
\feqn
Fix $\delta \in (0,1)$.  We next show that, when n is large
enough, the contribution of the first $[\delta n]$ summands on the righthand side
of \eqref{eq:display} is bounded by a continuous function of $\delta$ which vanishes at $0.$
We have
\beq
&&
\Bigl|\frac{1}{n\veps_n}\sum_{i=1}^{[\delta n]}
\log\Bigl(1+\fracd{\veps_i-\sqrt{1-(1-\veps_i^2)e^{2\rho_n^2 \lambda}}}{1-\veps_i}\Bigr)\Bigr|
\leq
\frac{1}{n\veps_n}\sum_{i=1}^{[\delta n]}
\fracd{(1+\veps_i)\bigl|e^{2\rho_n^2 \lambda}-1\bigr|}{
\veps_i+\sqrt{1-(1-\veps_i^2)e^{2\rho_n^2\lambda}}}
\\
&&
\qquad
\leq
\frac{2}{n\veps_n}\sum_{i=1}^{[\delta n]}
\fracd{\bigl|e^{2\rho_n^2 \lambda}-1\bigr|}{
\veps_i}\le \frac{2a_{_{[\delta n]}}}{n \veps_n}\bigl|e^{2\rho_n^2\lambda}-1\bigr|.
\feq
Since $(a_n)_{n\geq 1}\in \mbox{RV}(1+\alpha),$ Theorem~\ref{th:rv} implies that, as $n\to\infty,$
$a_{_{[\delta n]}}\sim \delta^{1+\alpha}a_n\sim \frac{\delta^{1+\alpha}}{1+\alpha}\veps_n^{-1}.$
Therefore, \beq \frac{2a_{_{[\delta n]}}}{n \veps_n}\bigl|e^{2\rho_n^2\lambda}-1\bigr| \underset{n\to\infty}{\sim}
\frac{2 \delta^{1+\alpha}\veps_n^{-1}}{(1+\alpha)n\veps_n}2\veps_n^2 \lambda
= \fracd{4 \lambda \delta^{1+\alpha}}{1+\alpha}\underset{\delta\to 0}{\longrightarrow 0}.
\feq
Hence,
\beq
\lim_{\delta\to0}\limsup_{n\to\infty} \Bigl|\frac{1}{n\veps_n}\sum_{i=1}^{[\delta n]}
\log\Bigl(1+\fracd{\veps_i-\sqrt{1-(1-\veps_i^2)e^{2\rho_n^2 \lambda}}}{1-\veps_i}\Bigr)\Bigr|=0.
\feq
Next, using elementary estimates on remainders of Taylor's series, we obtain
\beq
&&
\lim_{n\to\infty} \frac{1}{n\veps_n } \log E\bigl(e^{\lambda\rho_n^2  T_n}\bigr)
=
\lim_{\delta\to 0} \lim_{n\to\infty}
\frac{1}{n\veps_n}\sum_{i=[\delta n]}^n
\log\Bigl(1+\fracd{\veps_i-\sqrt{1-(1-\veps_i^2)e^{2\rho_n^2 \lambda}}}{1-\veps_i}\Bigr)
\\
&&
\qquad
=
\lim_{\delta\to 0} \lim_{n\to\infty}
\frac{1}{n\veps_n} \sum_{i=[\delta n]}^n
\fracd{(1+\veps_i)(e^{2\rho_n^2 \lambda}-1)}
{\veps_i+\sqrt{1-(1-\veps_i^2)e^{2\rho_n^2\lambda}}}
=
\lim_{\delta\to 0} \lim_{n\to\infty}
\frac{1}{n}\sum_{i=[\delta n]}^n
\fracd{2\lambda\rho_n}
{\veps_i+\sqrt{\veps_i^2-2\rho_n^2 \lambda}}
\\
&&
\qquad
=
\lim_{\delta\to 0} \lim_{n\to\infty}
\frac{1}{n}\sum_{i=[\delta n]}^n
\fracd{2\lambda}
{\veps_i/\rho_n+\sqrt{(\veps_i/\rho_n)^2-
2\lambda}}
=
\int_0^1 \fracd{2\lambda}{x^{-\alpha}+
\sqrt{x^{-2\alpha}-2\lambda}}~dx
=\Lambda(\lambda).
\feq
This completes the proof of \eqref{eq:expmom}.
\par
We note that $\lim_{\lambda \to -\infty}\Lambda(\lambda) = -\infty$.  In addition,
$$\Lambda'(\lambda)
=
\int_0^1
\bigl(x^{-2\alpha}-2\lambda\bigr)^{-1/2}~dx.
$$
This function is strictly increasing and hence $\Lambda$ is strictly
convex.  Note also that $\Lambda'(0)=\frac{1}{1+\alpha},$
$\lim_{\lambda \to -\infty} \Lambda'(\lambda)=0,$ and
$\lim_{\lambda\to \frac 12} \Lambda'(\lambda) =\infty$.
\par
For $z>0$, let $J_z(\lambda) = \Lambda(\lambda)-\lambda z/(1+\alpha)$.  This
function is convex and $J_z(0)=0$.  Since $J_z'(\lambda) =
\Lambda'(\lambda)-z/(1+\alpha),$ the minimum of $J_z$ is
uniquely attained at some $\lambda^* \in (-\infty,\frac 12),$ and
$J_z(\lambda^*)<0$ for $z\neq 1.$  In addition, if $z >1$, $\lambda^*>0$ and if
$z<1$, $\lambda^*<0$. \\
By Theorem \ref{th:rv}-(i), as $n\to\infty,$ $$a_n \rho_n^2 \sim \frac{n \veps_n^{-1} \veps_n^2}{1+\alpha}=\frac{n
\veps_n}{1+\alpha}.$$ It follows that if  $\lambda \in (0,\frac 12)$, then for $x>0,$ as $n\to \infty,$
\beq
&&
\fracd{1}{n\veps_n} \log P\bigl(T_n/a_n\geq 1+x)
\le \frac{1}{n\veps_n}\bigl[ \log E \bigl(e^{ \lambda
\rho_n^2 T_n}\bigr) - \lambda a_n \rho_n^2 (1+x)\bigr]\sim J_{1+x}(\lambda).
\feq
Therefore,
$$\limsup_{n\to\infty} \fracd{1}{n\veps_n} \log P\bigl(T_n/a_n\geq
1+x)\le \min_{0<\lambda < \frac 12} J_{1+x}(\lambda)<0.$$
If $\lambda <0$, then for $x\in (0,1),$  as $n\to\infty,$
\beq
\fracd{1}{n\veps_n} \log P\bigl(T_n/a_n\leq 1-x) \le \frac{1}{n\veps_n}\bigl[ \log
E \bigl(e^{ \lambda \rho_n^2 T_n}\bigr) - \lambda a_n \rho_n^2 (1-x) \bigr]\sim J_{1-x}(\lambda).
\feq
Therefore,
$$\limsup_{n\to\infty} \fracd{1}{n \veps_n } \log P\bigl(T_n/a_n\leq 1-x \bigr)
\leq \min_{\lambda < 0} J_{1-x} (\lambda)<0.$$
Moreover, since $\lim_{\lambda\to \frac 12} \Lambda'(\lambda) =\infty$, the log-generating function $\Lambda(\lambda)$ is steep in the terminology of
\cite{ldpb}. Therefore, by the G$\ddot{\mbox{a}}$rtner--Ellis theorem (cf.~p.~44~in~\cite{ldpb}, see also Remark~(a) following
the theorem), the above upper limits are in fact the limits. The proof of Proposition~\ref{ldp} is completed.
\end{proof}
\subsection{Proof of Theorem~\ref{walk}}
\label{subs:walk} Since the law of $X$ is symmetric about $0$, the
theorem is equivalent to the claim that $\lim_{n\to\infty} P(X_n
> xb_n )=e^{-2x}/2$ for all $x>0.$ Furthermore, since $\lim_{n\to\infty} b_n =\infty$
and $b_{n}\sim b_{n+1},$ it suffices to show that \beq
\lim_{n\to\infty} P(X_{2n} > xb_{2n} ) = e^{-2x}/2,\quad x>0. \feq
The idea of the proof is the following. In this subcritical regime,
we have seen in the beginning of the section that the number of
visits to the origin by time $2n$ is very-well localized around its
typical value $c_{2n}$ (cf Proposition~\ref{cor},
Corollaries~\ref{ldpeta} and \ref{weakdelta}). From properties of
regular varying sequences, this will imply that the drift at time
$2n$ is also very-well localized around its typical value
$\veps_{c_{2n}}$ (see assertions \eqref{eq:xilim} and
\eqref{eq:betaest} below). Then, by Lemma~\ref{lem:monot}, we are
able to compare our walk with oscillating walks with a drift close
to $\veps_{c_{2n}}$ (see \eqref{comp1} and \eqref{comp2} below), for
which we know the stationary distribution. In particular,
Lemma~\ref{lem:convrate} allows us to show that the distribution of
$X_n$ is close to that stationary distribution. Let us now turn to
the precise argument.
\par

Fix $x>0.$ We begin with an upper bound for $P(X_{2n} > xb_{2n} ).$
Recall the definition of $(\theta_n)$ from
Corollary~\ref{weakdelta}. For $n\geq 1,$ let \beq \Gamma_n =\{X_n >
xb_n ,~\eta_n \leq (1+\theta_n) c_n\}. \feq We have \beq P(X_{2n}
>xb_{2n}) \leq P(\Gamma_{2n})+P(\eta_{_{2n}}> (1+\theta_{2n}) c_{_{2n}}). \feq We
proceed with an estimate of the righthand side. By
Theorem~\ref{th:rv}-(ii),  as $n\to\infty,$
\begin{equation}
\label{eq:xilim} \xi_n:=\min_{ i\le (1+\theta_n)c_n}\veps_i \sim
\veps_{(1+\theta_n)c_n}\sim \veps_{c_{n}}.
\end{equation}
For $n\geq 1$ consider the sequence $\alpha_n=(\alpha_{n,k})_{k\ge 1}$
defined as follows: $\alpha_{n,k}= \veps_k$ for $k \le (1+\theta_n)
c_n$ and $\alpha_{n,k} = \xi_n$ for $k> (1+\theta_n)c_n$.
Since on event $\Gamma_n$ we have $\eta_n \le (1+\theta_n)c_n$, it follows that
$P(\Gamma_{2n})=P^{\alpha_{_{2n}}}(\Gamma_{2n})\le P^{\alpha_{_{2n}}}(X_{2n} > xb_{2n})$.
\par
Recall the notation $P^{(\delta)}$ introduced in the second
paragraph of Section~\ref{sec:prel} (this notation is distinct from
$P^\delta$ and emphasizes that the sequence $(\delta)$ is constant).
Since $\xi_n = \min_{k \ge 1} \alpha_{n,k}$, Lemma \ref{lem:monot}
implies: \beqn \label{comp1} \ \ P^{\alpha_{_{2n}}}(X_{2n} >
xb_{2n})\le P^{(\xi_{2n})}(X_{2n} > xb_{2n})\le
P^{(\xi_{2n})}_{\mu_{_{\xi_{2n}}}}(X_{2n} >xb_{2n})=
\mu_{_{\xi_{2n}}}\bigl((xb_{2n},\infty)\bigr). \feqn Therefore
\begin{equation}
\label{eq:ub} P(X_{2n} > xb_{2n}) \le \mu_{_{\xi_{2n}}}\bigl((xb_{2n},\infty)\bigr) + P(\eta_{_{2n}}>
(1+\theta_{2n})c_{_{2n}}).
\end{equation}
The second term on the righthand side of \eqref{eq:ub} converges to $0,$ as
$n\to\infty$, due to Proposition~\ref{cor}.
Furthermore, \eqref{eq:invmeas9} and \eqref{eq:xilim} yield that, as $n\to\infty,$
\beqn
\label{eq:lntail}
\mu_{_{\xi_{2n}}}\bigl((xb_{2n},\infty)\bigr)\sim  2 \xi_{2n}
\sum_{j=[xb_{2n}/2]}^{\infty}
\Bigl(\fracd{1-\xi_{2n}}{1+\xi_{2n}}\Bigr)^{2(j-1)}
\sim
\fracd{1}{2} \Bigl(\fracd{1-\xi_{2n}}{1+\xi_{2n}}\Bigr)^{xb_{2n}} \underset{\rho \to 0}{\longrightarrow} \fracd{1}{2}e^{-2x}.
\feqn
Using \eqref{eq:ub}, we conclude that
\beq
\limsup_{n\to\infty} P(X_{2n}> xb_{2n}) \leq e^{-2x}/2.\feq
We now turn to a lower bound on $P(X_{2n}> xb_{2n}).$ It follows from Corollary~\ref{sequen}-(iv)
that there exists a sequence $(\kappa_n)_{n\ge 1}$ taking values in
$2\zz_+$ and satisfying
\begin{align}
\label{eq:kappa}
\lim_{n\to\infty} \kappa_n/n=0 \quad \mbox{and} \quad
\lim_{n\to\infty} \frac{\kappa_n \veps_{c_{_n}}^2}{\log (\veps_{c_{_n}}^{-1})}=\infty.
\end{align}
Note that the second limit in \eqref{eq:kappa} ensures that $\lim_{n\to\infty}
\kappa_n=\infty$.  Let
\beqn
\label{calmn}
\Upsilon_n=\{m\in \nn: |m-c_n|\leq \theta_n c_n\}.
\feqn
By Theorem~\ref{th:rv}-(ii), we have, as $n\to\infty,$
\begin{equation}
\label{eq:betaest} \beta_n:= \max_{m\in \Upsilon_n } \veps_m \sim \veps_{c_{_n}}.
\end{equation}
Since the function $z\to z^2 / \log (z^{-1})$ is increasing on
$(0,1)$, the second limit in \eqref{eq:kappa} along with \eqref{eq:betaest}
imply that $\lim_{n\to\infty} \frac{\kappa_n \beta_n^2}{\log
(\beta_n^{-1})}=\lim_{n\to\infty} \frac{\kappa_n \beta_{n-\kappa_n}^2}{\log
(\beta_{n-\kappa_n}^{-1})}=\infty$. Therefore,
\begin{equation}
\label{eq:makesure} \lim_{n\to\infty}
(1+\beta_n^2)^{-\kappa_n}\beta_n^{-s}=\lim_{n\to\infty}
(1+\beta_{n-\kappa_n}^2)^{-\kappa_n}\beta_{n-\kappa_n}^{-s}=0~\mbox{for all}~s \in \rr.
\end{equation}
We have
\begin{align}
\nonumber
& P(X_{2n} >xb_{2n} )=\frac{1}{2} P(|X_{2n}| >xb_{2n} )\ge \frac{1}{2} P(|X_{2n}| >
xb_{2n},\eta_{_{2n-\kappa_{_{2n}}}}\in \Upsilon_{2n-\kappa_{_{2n}}})\\
\nonumber & \qquad =\frac{1}{2} \sum_{  m \in \Upsilon_{2n-\kappa_{_{2n}}}}
P(|X_{2n}|> xb_{2n},\eta_{_{2n-\kappa_{_{2n}}}}=m)\\
\label{eq:hamechaye}&\qquad =\frac{1}{2}\sum_{
m \in \Upsilon_{2n-\kappa_{_{2n}}}}
\sum_{j\in 2\zz}  E \bigl
(\Ind_{\{\eta_{_{2n-\kappa_{_{2n}}}}=m,X_{2n-\kappa_{_{2n}}}=j\}}
P_{(j,m)}(|X_{\kappa_{_{2n}}}| > xb_{2n})\bigr).
\end{align}
For $j\in 2\zz$ and $m\in \Upsilon_{2n-\kappa_{_{2n}}},$
Lemma~\ref{lem:monot} implies that \beqn
\label{comp2}
P_{(j,m)}(|X_{\kappa_{_{2n}}} |
> xb_{2n}) \geq
P^{(\beta_{2n-\kappa_{_{2n}}})}_j(|X_{\kappa_{_{2n}}}|> xb_{2n}
)\geq P^{(\beta_{2n-\kappa_{_{2n}}})}(|X|_{\kappa_{_{2n}}}>
xb_{2n}).\feqn Plugging this inequality into the righthand side of
\eqref{eq:hamechaye}, we obtain
\begin{equation}
\label{eq:almostdone} P(X_{2n} >xb_{2n})  \ge
P(\eta_{_{2n-\kappa_{_{2n}}}}\in \Upsilon_{2n-\kappa_{_{2n}}})
P^{(\beta_{2n-\kappa_{_{2n}}})}(X_{_{\kappa_{2n}}}> xb_{2n}).
\end{equation}
The first term on the righthand side of \eqref{eq:almostdone} converges to $1,$ as
$n\to\infty,$ by Corollary~\ref{weakdelta}. Moreover, by Lemma~\ref{lem:convrate},
$$ P^{(\beta_{2n-\kappa_{_{2n}}})}(X_{\kappa_{_{2n}}}> xb_{2n} )
\ge \mu_{_{\beta_{2n-\kappa_{_{2n}}}}}\bigl((xb_{2n},\infty)\bigr) -
2(1+\beta_{2n-\kappa_{_{2n}}}^2)^{-\kappa_{_{2n}}}.$$
The second term on the righthand side converges to $0$ due to \eqref{eq:makesure}. Therefore,
\eqref{eq:lntail} and \eqref{eq:betaest} imply that
$$\liminf_{n\to\infty} P(X_{2n} > xb_{2n}) \geq e^{-2x}/2,$$
which completes the proof of Theorem~\ref{walk}. \qed
\subsection{Proof of Theorem \ref{lastvislastexc}}
\label{last}
\paragraph{\bf Proof of Theorem~\ref{lastvislastexc}-(i)}
As in the previous paragraph, this proof once again relies on
Lemma~\ref{lem:convrate} and Corollary~\ref{ldpeta}. We adopt
notation from the proof of Theorem~\ref{walk} above.
\par
It follows from \eqref{eq:ub} that
$$P(X_{2n}=0) \geq \mu_{_{\xi_{2n}}}(0)-P(\eta_{2n}\ge (1+\theta_{2n})c_{_{2n}}).$$
Therefore,
$$b_{2n} P(X_{2n}=0)\underset{\eqref{eq:invmeas9}}{\ge}
\frac{2}{1+\xi_{2n}}\frac{\xi_{2n}}{\veps_{c_{_{2n}}}}-b_{2n}
P(\eta_{2n}\ge (1+\theta_{2n})c_{_{2n}}).$$ The second term on the righthand
side converges to $0$ due to Corollary~\ref{ldpeta} while he first term
converges to $2$ due to \eqref{eq:xilim}. Hence,
$$ \liminf_{n\to\infty} b_{2n} P(X_{2n}=0)\ge 2.$$
The upper bound is obtained in a similar way. Recall $\Upsilon_n$
was defined in \eqref{calmn}. By \eqref{eq:almostdone},
\begin{align*}
P(X_{2n}=0)& \le 1-P(\eta_{_{2n-\kappa_{_{2n}}}}\in \Upsilon_{2n-\kappa_{_{2n}}})
P^{(\beta_{2n-\kappa_{_{2n}}})}(X_{\kappa_{2n}}\ge 2)\\
& \le 1-(1-P(\eta_{_{2n-\kappa_{_{2n}}}}\not\in \Upsilon_{2n-\kappa_{_{2n}}})P^{(\beta_{2n-\kappa_{_{2n}}})}(X_{\kappa_{_{2n}}}\geq 2)\\
& = P^{(\beta_{2n-\kappa_{_{2n}}})}(X_{\kappa_{_{2n}}}=0)+P(\eta_{_{2n-\kappa_{_{2n}}}}\not\in \Upsilon_{2n-\kappa_{_{2n}}})\\
&\le \mu_{_{\beta_{2n-\kappa_{_{2n}}}}}(0)+2(1+\beta_{2n-\kappa_{_{2n}}}^2)^{-\kappa_{_{2n}}}+
P(\eta_{_{2n-\kappa_{_{2n}}}}\not\in \Upsilon_{2n-\kappa_{_{2n}}}),
\end{align*}
where in the last step we used Lemma~\ref{lem:convrate}. Therefore,
\beq
b_{2n} P(X_{2n}=0)& \underset{\eqref{eq:invmeas9}}{\le}
\fracd{2}{1+\beta_{2n-\kappa_{_{2n}}}}\fracd{\beta_{2n-\kappa_{_{2n}}}}
{\veps_{c_{_{2n}}}}+2(1+\beta_{2n-\kappa_{_{2n}}}^2)^{-\kappa_{2n}}\beta_{2n-\kappa_{_{2n}}}^{-1}
\fracd{\beta_{2n-\kappa_{_{2n}}}}{\veps_{c_{2n}}}\\ &+b_{2n}P(\eta_{_{2n-\kappa_{_{2n}}}}\not\in \Upsilon_{2n-\kappa_{_{2n}}}).
\feq
The third term on the righthand side converges to
$0$ due to Corollary \ref{ldpeta}. The second term on the righthand side converges to $0$ by
\eqref{eq:betaest} and \eqref{eq:makesure}. Finally, the first term
on the righthand side converges to $2$ by \eqref{eq:betaest}. Hence,
$$\limsup_{n\to\infty} b_{2n} P(X_{2n} =0)\le 2.$$
This completes the proof of the first part of Theorem~\ref{lastvislastexc}.
\qed
\\
$\mbox{}$
\paragraph{\bf Proof of Theorem \ref{lastvislastexc}-(ii)}
\label{lastagain}
Recall $\Upsilon_n$ from \eqref{calmn} and $\Lambda_n$ from \eqref{positive}.
By Corollary~\ref{weakdelta}, and using the Markov property, we
obtain for $t>0,$
\beqn
\nonumber
&& \liminf_{n\to\infty} b_{2n}^2 P(V_{2n}=2n-2[tb_{2n}^2])
\\
\nonumber
&& \qquad \qquad =\liminf_{n\to\infty} b_{2n}^2 \sum_{m\in \Upsilon_{2n-2[tb_{2n}^2]}} P(V_{2n}=2n-2[tb_{2n}^2], \eta_{2n-2[tb_{2n}^2]}=m)
\\
\label{llimit}
&&
\qquad \qquad =\liminf_{n\to\infty} b_{2n}^2 \sum_{m\in \Upsilon_{2n-2[tb_{2n}^2]}}
P(X_{2n-2[tb_{2n}^2]}=0,\eta_{2n-2[tb_{2n}^2]}=m)\cdot 2P^{(\veps_m)}(\Lambda_{2[tb_{2n}^2]}),
\feqn
The factor 2 in the last line comes from the fact that we also want count excursions to the negative half-line
and a symmetry argument.
\par
Recall \eqref{eq:betaest}. By Lemma~\ref{lem:monot},
\beq
&& \liminf_{n
\to \infty} b_{2n}^2 P(V_{2n}=2n-2[tb_{2n}^2]) \\ \nonumber && \qquad \geq
2\liminf_{n \to \infty} b_{2n}^2 P\bigl( X_{2n-2[tb_{2n}^2]}=0,
\eta_{2n-2[tb_{2n}^2]} \in \Upsilon_{2n-2[tb_{2n}^2]} \bigr) P^{(\beta_{2n-2[tb_{2n}^2]})}(\Lambda_{2[tb_{2n}^2]}).
\feq
Using again Corollary~\ref{weakdelta}, and taking in account that
$\lim_{n\to\infty} b_{2n}/b_{2n-2[tb_{2n}^2]}=1,$ we get
\beq
\liminf_{n \to \infty} b_{2n}^2 P(V_{2n}=2n-2[tb_{2n}^2]) \geq
2 \liminf_{n \to \infty} b_{2n}^2 P( X_{2n-2[tb_{2n}^2]}=0) P^{(\beta_{2n-2[tb_{2n}^2]})}(\Lambda_{2[tb_{2n}^2]}).
\feq
Using Lemma~\ref{meander}-(i) and Theorem~\ref{lastvislastexc}-(i), we conclude that
\beq
\liminf_{n \to \infty} b_{2n}^2 P(V_{2n}=2n-2[tb_{2n}^2]) &\geq&
2\cdot \fracd{1}{\sqrt{2t}}\int_0^{\infty} du \
\frac{2u}{\sqrt{2\pi }} \exp \Bigr(-\frac{u^2}{2} - u \sqrt{2t} -t \Bigl)
\\
&=& \sqrt{2} \mn^{(-\sqrt{2})}(\zeta>t).
\feq
A very similar argument shows that
\beq
\limsup_{n \to \infty} b_{2n}^2 P(V_{2n}=2n-2[tb_{2n}^2]) \leq  \sqrt{2} \mn^{(-\sqrt{2})}(\zeta>t),
\feq
from which Theorem~\ref{lastvislastexc}-(ii) follows in view of \eqref{scaling}.
\\
$\mbox{}$
\paragraph{\bf Proof of Theorem~\ref{lastvislastexc}-(iii).}
Fix a bounded continuous function $F : C(\rr_+,\rr) \to \rr$, a
constant $t>0.$ Let $Z_n$ be the process defined in the statement
of the theorem and let $\bigl(\witi Z_n(t)\bigr)_{t\in \rr_+}$ be a continuous process for which
$\witi Z_n\bigl(k/b_{2n}^2\bigr)=|X_{k\wedge T_1}|/b_{2n}$ whenever $k \in \zz_+,$ and which is
linearly interpolated elsewhere.
\par
For $n$ large enough, so that the quantities below are well defined,
the Markov property implies that \beq E\bigl(F(Z_n)| V_{2n} = 2n -
2[tb_{2n}^2],\eta_{2n}=m\bigr) = E^{(\veps_m)} \bigl(F(\witi
Z_n)\bigl| \Lambda_{2[tb_{2n}^2]} \bigl). \feq Let \beqn
\label{eihn} H_n(t):=b_{2n}^2\sum_{m\in \Upsilon_{2n}} E^{(\veps_m)}
\bigl(F(\witi Z_n)\bigl| \Lambda_{2[tb_{2n}^2]} \bigl)P\bigl(V_{2n}
= 2n - 2[tb_{2n}^2],\eta_{2n}=m\bigr). \feqn Since $F$ is bounded,
Corollary~\ref{weakdelta} implies that \beqn \label{getthedrift}
\lim_{n \to \infty} \bigl| b_{2n}^2 E\bigl(F(Z_n);V_{2n} = 2n -
2[tb_{2n}^2]\bigr)  - H_n(t) \bigr| =0. \feqn Recall $\Upsilon_n$
was defined in \eqref{calmn}. For $m\in\nn,$ let $\underline m_n \in
\Upsilon_{2n}$ and $\overline m_n\in \Upsilon_{2n}$ be such that
\beq E^{(\veps_{_{\underline m_n}})} \bigl(F(\witi Z_n)\bigl|
\Lambda_{2[tb_{2n}^2]} \bigl)=\min_{m\in \Upsilon_{2m}}E^{(\veps_m)}
\bigl(F(\witi Z_n)\bigl| \Lambda_{2[tb_{2n}^2]} \bigl) \feq and \beq
E^{(\veps_{_{\overline m_n}})} \bigl(F(\witi Z_n)\bigl|
\Lambda_{2[tb_{2n}^2]} \bigl)=\max_{m\in \Upsilon_{2m}}E^{(\veps_m)}
\bigl(F(\witi Z_n)\bigl| \Lambda_{2[tb_{2n}^2]} \bigl). \feq Since
$\lim_{n\to\infty} \veps_{_{\underline m_n}}
b_{2n}=\lim_{n\to\infty} \veps_{_{\overline m_n}} b_{2n}=1,$
Lemma~\ref{meander} implies that there exists \beqn \label{vaybar}
\lim_{n \to \infty} E^{(\veps_{_{\underline m_n}})}\bigl(F(\witi
Z_n)\bigl| \Lambda_{2[tb_{2n}^2]}\bigr) =\lim_{n \to \infty}
E^{(\veps_{_{\overline m_n}})}\bigl(F(\witi Z_n)\bigl|
\Lambda_{2[tb_{2n}^2]}\bigr) =E\bigl(F(\ol Y)\bigr) \feqn where $\ol
Y=(\ol Y(s))_{s\in \rr_+}$ is a non-negative process in
$C(\rr_+,\rr)$ such that $\fracd{1}{\sqrt{2t}}\bigl(\ol
Y(2ts)\bigr)_{s\in \rr_+}$ is distributed according to the law
$\mn^{(-\sqrt{2t})}(~\cdot~|\zeta>1),$ and the underlying
probability space is enlarged, if needed, to include this process.
The scaling property \eqref{scaling} implies that $\ol Y$ is
distributed according to the law $\mn^{(-1)}(~\cdot~|\zeta>2t).$
\par
In virtue of Theorem~\ref{lastvislastexc}-(ii) and
Corollary~\ref{weakdelta}, the claim of Theorem~\ref{lastvislastexc}-(iii)
follows from the above convergence and \eqref{eihn}. \qed
\begin{remark}
\label{alter} Theorem~\ref{lastvislastexc}-(ii) along with
\eqref{vaybar} yield \beq \lim_{n \to \infty} P(|X_{2n}|>xb_{2n})=
\int_{0}^{\infty}\!\!\! dt~\mn^{(-1)}(X_t >x, \zeta >t) =\int_x^\infty \!\!\! dy
\int_0^\infty \! \! \! dt \fracd{2y}{\sqrt{2\pi t^3}}\exp\Bigl(-\frac{(y-ct)^2}{2t}\Bigr). \feq It is not hard to check
that the right-hand side above is $\exp(-2x)$ in agreement with
Theorem~\ref{walk}.
\end{remark}
\subsection{Proof of Theorem~\ref{supt} and Corollary \ref{cor:supt}}
\label{supr}
\begin{proof}[Proof of Theorem \ref{supt}]
For $i\ge 1$ let $S_i = \max_{T_{i-1} \le k < T_{i}} |X_k|$. For $x>0$ let $x_n =x h_n,$ where
$h_n$is defined in the statement of the theorem. Recall $(\theta_n)_{n\geq 1}$ from Corollary~\ref{weakdelta}
and $\Upsilon_n$ from \eqref{calmn}.
\par
Fix any $x\in (0,\infty)\backslash\{1\},$ $\lambda\in (0,1),$ and
assume that $n\in\nn$ below is large enough, so that
$1-\theta_n>\lambda.$ Then, on one hand, \beqn \nonumber &&
P\bigl(|M_n| \leq x_n \bigr)= P\bigl(|M_n| \leq x_n,\eta_n
<c_n(1-\theta_n)\bigr) + P\bigl(|M_n| \leq x_n,\eta_n \geq
c_n(1-\theta_n) \bigr)
\\
\label{hand}
&&
\qquad
\leq
P(\eta_n\not\in\Upsilon_n)+\prod_{i=[\lambda c_n]}^{[c_n(1-\theta_n)]}P\bigl(S_i \leq x_n\bigr),
\feqn
and on the other hand,
\beqn
\nonumber
&&
P\bigl(|M_n| \leq x_n \bigr) \geq P\bigl(|M_n| \leq x_n,\eta_n \leq c_n(1+\theta_n)\bigr)
\\
\label{hand1} && \qquad \geq
-P(\eta_n\not\in\Upsilon_n)+\prod_{i=1}^{[c_n(1+\theta_n)]}
P\bigl(S_i \leq x_n\bigr). \feqn Observe now that \beq
\label{eq:niceandgood} \lim_{n\to\infty} \frac{c_n\veps_{c_{n}}}{n
\veps_{c_{n}}^2} \underset{\mbox{Theorem }\ref{th:rv}-(iii)}{=}
\lim_{n\to\infty} \frac{c_n
}{a_{c_n}\veps_{c_{n}}}\underset{\mbox{Theorem }\ref{th:rv}-(i)}{=}
(1+\alpha) <\infty, \feq and hence, by Corollary~\ref{weakdelta},
for all $z>0,$  \beqn \label{finale} \lim_{n\to\infty}
\fracd{P(\eta_n\not\in\Upsilon_n)} {e^{-(c_n \veps_{c_{n}})^z}}=0.
\feqn Next, by the well-known formula for the ruin probability (see
for instance \cite[p.~274]{durrett}),
\begin{equation}
\label{eq:alternative}
P(S_i \le x_n)= 1- \frac{\rho_i}{(1+\rho_i)^{x_n}-1},
\end{equation}
where $\rho_i =\fracd {2 \veps_i}{1-\veps_i}.$ For $n\in \nn$ let
\begin{equation}
\label{notka}
\chi_n: = \min_{1 \leq i \leq (1+\theta_n)c_n} \rho_i \sim 2\veps_{c_{_n}} \quad \mbox{and}
\quad \beta_{n,\lambda} := \max_{\lambda c_n\leq i \leq (1+\theta_n)c_n}
\rho_i  \sim 2 \lambda^{-\alpha}\veps_{c_{_n}},
\end{equation}
where we use Theorem~\ref{th:rv}-(iii) to state the equivalence
relations. Since the righthand side in \eqref{eq:alternative} is an
increasing function of $\rho_i$, we obtain: \beq \log
\prod_{i=1}^{[(1+\theta_n)c_n]} P(S_i \le x_n) &\ge
{[(1+\theta_n)c_n]}\log \left( 1 -
\frac{\chi_n}{(1+\chi_n)^{x_n}-1}) \right) \underset{n\to\infty}{\sim}
-\fracd{c_n\chi_n}{(1+\chi_n)^{x_n}}. \feq We next estimate the
rightmost expression above. Using \eqref{notka} and the definition
of $h_n$ given in the statement of Theorem~\ref{supt}, we have, as
$n\to\infty,$ \beq \fracd{1}{\log (\veps_{c_{_n}}c_n)}\cdot \log
\fracd{c_n\chi_n}{(1+\chi_n)^{x_n}}\sim
1-\fracd{2x\veps_{c_{_n}}h_n}{\log (\veps_{c_{_n}}c_n)}\sim 1-x,
\feq Similarly, as $n\to \infty,$ \beq \log \prod_{i=[\lambda
c_n]}^{[(1-\theta_n)c_n]} P(S_i \le x_n) &\le
{[(1-\theta_n-\lambda)c_n]}\log \left( 1 -
\fracd{\beta_{n,\lambda}}{(1+\beta_{n,\lambda})^{x_n}-1} \right)
\sim - \fracd{(1-\lambda)
c_n\beta_{n,\lambda}}{(1+\beta_{n,\lambda})^{x_n}}, \feq and \beq
\fracd{1}{\log (\veps_{c_{_n}}c_n)}\cdot \log \fracd{(1-\lambda)
c_n\beta_{n,\lambda}}{(1+\beta_{n,\lambda})^{x_n}} \sim
1-\fracd{2x\lambda^{-\alpha}\veps_{c_{_n}}h_n}{\log
(\veps_{c_{_n}}c_n)}\sim 1-x\lambda^{-\alpha}, \feq Since
$\lambda\in (0,1)$ is arbitrary, we conclude from \eqref{hand},
\eqref{hand1}, and \eqref{finale} that \beq \lim_{n\to\infty}
\frac{1}{\log(c_n \veps_{c_{n}})} \log\bigl(-\log P(|M_n| \leq
x_n)\bigr)=1-x. \feq Note that if $x>1,$ this is equivalent to $
\lim_{n\to\infty} \frac{1}{\log(c_n \veps_{c_{n}})}
\log\bigl(P(|M_n| >x_n)\bigr)=1-x.$
\par
To complete the proof of Theorem~\ref{supt}, observe that \beq
P\bigl(\max_{T_{i-1} \le k < T_{i}} X_k \leq x_n\bigr
)=\fracd{1}{2}+\fracd{1}{2}P\bigl(S_i\leq x_n\bigr )=1-\fracd{1}{2}
\fracd{\rho_i}{(1+\rho_i)^{x_n}-1}.\feq Therefore, replacing $|M_n|$
with $M_n$ and $S_i$ with $\max_{T_{i-1} \le k < T_{i}} X_k$ in
\eqref{hand} and \eqref{hand1}, the proof given above for $|M_n|$
goes through verbatim for $M_n.$
\end{proof}
\begin{proof}[Proof of Corollary \ref{cor:supt}]
Theorem~\ref{supt} implies $\lim_{n\to\infty} M_n / h_n=1$ in
probability. Furthermore, by Corollary \ref{sequen}-(iii),
$\veps_{c_{n}} c_n \in \mbox{RV}((1-\alpha)/(1+\alpha)).$ Therefore, if $\alpha<1,$
Theorem~\ref{supt} implies that for any $x>0$ there exists a constant
$z=z(x)>0$ such that \beq P\bigl(|M_n-h_n|>xh_n\bigr) \leq n^{-z}
\feq for all $n$ sufficiently large.
\par
Once this point is reached, the rest of the
proof is standard (see for instance \cite[Section~1.7]{durrett}).
Fix $\gamma>1$ and let $m_n=[\gamma^n].$ Using the Borel-Cantelli lemma, we obtain that \beq
P\bigl(|M_{m_n}-h_{m_n}|>xh_{m_n}~\mbox{i.o.}\bigr)=0,\qquad x>0.
\feq
Therefore $ \lim_{n\to\infty} M_{m_n}/h_{m_n}=1,$ $\as$ Moreover, if
$m_n\leq k<m_{n+1},$
\beq
\fracd{M_{m_n}}{h_{m_n}}\fracd{h_{m_n}}{h_k} \leq \fracd{M_k}{h_k} \leq \fracd{M_{m_{n+1}}}
{h_{m_{n+1}}}\fracd{h_{m_{n+1}}}{h_k}.
\feq
Since $\lim_{n\to\infty} m_{m+1}/m_m =\gamma$ and $(h_n)_{n\ge 1}\in
\mbox{RV}(\alpha/(1+\alpha)),$ Theorem \ref{th:rv}-(ii) implies that
\beq
\gamma^{-\frac{\alpha}{1+\alpha}} \leq \liminf_{k\to\infty} \fracd{M_k}{h_k} \leq
\limsup_{k\to\infty} \fracd{M_k}{h_k} \leq \gamma^{\frac{\alpha}{1+\alpha}}, \qquad P-\mbox{a.s.}
\feq
Since $\gamma>1$ is arbitrary, it follows that $\lim_{k\to\infty} M_k /h_k =1,$ $P$-a.s.
Furthermore, if $(k_n)_{n\geq 1}$ is a random sequence such that $X_{k_n} = M_{k_n},$ we have:
\beq
\limsup_{n\to\infty} \fracd{X_n}{h_n} \geq \limsup_{n\to\infty} \frac{X_{k_n}}{h_{k_n}}=\lim_{n\to\infty}
\frac{M_{k_n}}{h_{k_n}}=1,
\feq
where the limits hold $P$-a.s. when $\alpha<1$ and in probability when $\alpha=1.$ Since $X_n\leq M_n,$ this
finishes the proof of the corollary.
\end{proof}
\section*{Acknowledgments}
We would like to express our gratitude to all the people with whom we were discussing various aspects
of this paper. We wish to especially thank David Brydges and Emmanuel Jakob with whom
the second and the third named authors got used to share their ideas during the work on this project.
\providecommand{\bysame}{\leavevmode\hbox to3em{\hrulefill}\thinspace}

\end{document}